\documentclass[a4paper,oneside,11pt]{article}
\usepackage[active]{srcltx}
\usepackage{hyperref}
\usepackage{authblk}
\usepackage{fancyhdr}
\usepackage{bibspacing}
\usepackage{makecell}

\newlength{\defbaselineskip}
\newcommand{\setlinespacing}[1]           {\setlength{\baselineskip}{#1 \defbaselineskip}}

\usepackage{amsmath}
\usepackage{amsthm}
\usepackage{amssymb}
\usepackage{latexsym}
\usepackage{mathtools}
\usepackage{mathrsfs}
\usepackage{graphics}
\usepackage{latexsym}
\usepackage{psfrag}
\usepackage{import}
\usepackage{verbatim}
\usepackage{enumerate}
\usepackage{enumitem}
\usepackage{graphicx}
\usepackage[usenames]{color}

\theoremstyle{plain}
\newtheorem{lemma}{Lemma}[section]
\newtheorem{theorem}[lemma]{Theorem}

\newtheorem{proposition}[lemma]{Proposition}
\newtheorem{corollary}[lemma]{Corollary}
\newtheorem*{theorem*}{Theorem}
\newtheorem*{maintheorem*}{Main Theorem}

\theoremstyle{definition}

\newtheorem{definition}[lemma]{Definition}
\newtheorem*{definition*}{Definition}
\newtheorem{remark}[lemma]{Remark}

\numberwithin{equation}{section}

\makeatletter
\newcommand{\mms}{m.m.s\@ifnextchar.{}{\@ifnextchar'{.}{\@ifnextchar){.}{.~}}}}
\makeatother

\usepackage{accents}
\newcommand{\ubar}[1]{\underaccent{\bar}{#1}}

\newcommand{\R}{\mathbb{R}}

\newcommand{\B}{\mathcal{B}}
\newcommand{\Q}{\mathcal{Q}}
\newcommand{\T}{\mathcal{T}}

\newcommand{\C}{\mathbb{C}}
\newcommand{\D}{\mathcal{D}}

\newcommand{\Haus}{\mathcal{H}}

\renewcommand{\L}{\mathcal{L}}
\newcommand{\F}{\mathcal{F}}
\newcommand{\RCD}{\mathsf{RCD}}
\newcommand{\CD}{\mathsf{CD}}
\newcommand{\QCD}{\mathsf{QCD}}

\newcommand{\RQCD}{\mathsf{RQCD}}
\newcommand{\CGTD}{\mathsf{CGTD}}
\newcommand{\RCGTD}{\mathsf{RCGTD}}

\newcommand{\geo}{{\rm Geo}}
\newcommand{\MCP}{\mathsf{MCP}}

\newcommand{\OptGeo}{{\mathrm{OptGeo}}}
\newcommand{\Opt}{{\mathrm{Opt}}}
\newcommand{\blambda}{\lambda}  
\newcommand{\Id}{{\rm Id}}
 
\newcommand{\diam}{\mathop{\rm diam}\nolimits} 
\newcommand{\supp}{\mathop{\rm supp}\nolimits}   \newcommand{\conv}{\mathop{\rm geo}\nolimits}      

\newcommand{\e}{{\rm e}}
\renewcommand{\r}{{\rm r}}

\DeclareMathOperator*{\relint}{relint}
\DeclareMathOperator*{\intt}{int}

\newcommand{\norm}[1]{\left\Vert#1\right\Vert}

\newcommand{\abs}[1]{\left\vert#1\right\vert}
\newcommand{\set}[1]{\left\{#1\right\}}
\newcommand{\brac}[1]{\left(#1\right)}

\newcommand{\Real}{\mathbb{R}}

\newcommand{\eps}{\varepsilon}

\renewcommand{\L}{\mathcal{L}}

\newcommand{\vol}{\textrm{Vol}}

\renewcommand{\P}{\mathbb P}

\renewcommand{\P}{\mathcal{P}}

\newcounter{mycounter}
\renewcommand{\H}{\mathbb{H}}

\newcommand{\mm}{\mathfrak m}
\newcommand{\qq}{\mathfrak q}

\newcommand{\ee}{{\rm e}}
\newcommand{\QQ}{\mathfrak Q}

\renewcommand{\d}{\mathsf d}

\providecommand{\keywords}[1]{\textbf{\textit{Keywords:}} #1}

\author{
Emanuel Milman\thanks{Mathematics Department, Technion - I.I.T., Haifa 32000, Israel. Email: emilman@tx.technion.ac.il. \\
\textrm{The research leading to these results is part of a project that has received funding from the European Research Council (ERC) under the European Union's Horizon 2020 research and innovation programme (grant agreement No 637851).}
}}

\date{}                     

\title{The Quasi Curvature-Dimension Condition\\with applications to sub-Riemannian manifolds}

\begin{document}
\pagestyle{fancy}
\thispagestyle{empty}

\maketitle

\bibliographystyle{plain}

\begin{abstract}
We obtain the best known quantitative estimates for the $L^p$-Poincar\'e and log-Sobolev inequalities on domains in various sub-Riemannian manifolds, including ideal Carnot groups and in particular ideal generalized H-type Carnot groups and the Heisenberg groups, corank $1$ Carnot groups, the Grushin plane, and various H-type foliations, Sasakian and $3$-Sasakian manifolds. Moreover, this constitutes the first time that a quantitative estimate independent of the dimension is established on these spaces. For instance, the Li--Yau / Zhong--Yang spectral-gap estimate holds on all Heisenberg groups of arbitrary dimension up to a factor of $4$.

We achieve this by introducing a quasi-convex relaxation of the Lott--Sturm--Villani $\CD(K,N)$ condition we call the Quasi Curvature-Dimension condition $\QCD(Q,K,N)$.
Our motivation stems from a recent interpolation inequality along Wasserstein geodesics in the ideal sub-Riemannian setting due to Barilari and Rizzi.
We show that on an ideal sub-Riemannian manifold of dimension $n$, the Measure Contraction Property $\MCP(K,N)$ implies $\QCD(Q,K,N)$ with $Q = 2^{N-n} \geq 1$, thereby verifying the latter property on the aforementioned ideal spaces; a result of Balogh--Krist\'aly--Sipos is used instead to handle non-ideal corank $1$ Carnot groups. 
 By extending the localization paradigm to completely general interpolation inequalities, we reduce the study of various analytic and geometric inequalities on $\QCD$ spaces to the one-dimensional case. 
Consequently, we deduce that while (strictly) sub-Riemannian manifolds do not satisfy any type of $\CD$ condition, many of them satisfy numerous functional inequalities with \emph{exactly the same} quantitative dependence (up to a factor of $Q$) as their $\CD$ counterparts. 
\end{abstract}


\section{Introduction}

The Curvature-Dimension condition $\CD(K,N)$ was first introduced in the 1980's by Bakry and \'Emery \cite{BakryEmery,BakryStFlour} in the context of diffusion generators, having in mind primarily the setting of weighted Riemannian manifolds, namely smooth Riemannian manifolds endowed with a smooth density with respect to the Riemannian volume. The $\CD(K,N)$ condition serves as a generalization of the classical condition in the non-weighted Riemannian setting of having Ricci curvature bounded below by $K \in \Real$ and dimension bounded above by $N \in [1,\infty]$ (see e.g. \cite{EMilmanNegativeDimension,Ohta-NegativeN} for further possible extensions). Numerous consequences of this condition 
have been obtained over the past decades, extending results from the classical non-weighted setting and at times establishing new ones directly in the weighted one. These include diameter bounds, volume comparison theorems, heat-kernel and spectral estimates, Harnack inequalities, topological implications, Brunn--Minkowski-type inequalities, and isoperimetric, functional and concentration inequalities -- see e.g. \cite{Ledoux-Book,BGL-Book,VillaniOldAndNew} and the references therein. 

\smallskip
Being a differential and Hilbertian condition, it was for many years unclear how to extend the Bakry--\'Emery definition beyond the smooth Riemannian setting. A satisfactory definition was finally found based on the theory of Optimal Transport \cite{AmbrosioLectureNotesOT,AmbrosioGigli-UsersGuide,EvansSurveyOT,McCannGuillenLectureNotes,RachevRuschendorf-Book,UrbasLectureNotesOT,VillaniTopicsInOptimalTransport,VillaniOldAndNew}. Given two probability measures $\mu_0,\mu_1$ on a common geodesic space $(X,\d)$ and a prescribed cost of transporting a single mass from point $x$ to $y$, the Monge-Kantorovich idea is to optimally couple $\mu_0$ and $\mu_1$ by minimizing the total transportation cost, and as a byproduct obtain a Wasserstein geodesic $[0,1] \ni t \mapsto \mu_t$ connecting $\mu_0$ and $\mu_1$ in the space of probability measures $\P(X)$. This gives rise to the notion of displacement convexity of a given functional on $\P(X)$ along Wasserstein geodesics, introduced and studied by McCann \cite{McCannConvexityPrincipleForGases}. Following the works of Cordero-Erausquin--McCann--Schmuckenschl{\"a}ger \cite{CMSInventiones}, Otto--Villani \cite{OttoVillaniHWI} and von Renesse--Sturm \cite{VonRenesseSturmRicciChar}, it was realized that the $\CD(K,\infty)$ condition in the smooth setting may be equivalently formulated synthetically as a certain convexity property of an entropy functional along $W_2$ Wasserstein geodesics (associated to $L^2$-Optimal-Transport, when the transport-cost is given by the squared-distance function).

\smallskip
This idea culminated in the seminal works of Lott, Sturm and Villani \cite{LottVillaniGeneralizedRicci,SturmCD1,SturmCD2}, where a synthetic definition of $\CD(K,N)$ was proposed on a general (complete, separable) metric space $(X,\d)$ endowed with a (locally-finite Borel) reference measure $\mm$ (``metric-measure space"); it was moreover shown that the latter definition coincides with the Bakry--\'Emery one in the smooth Riemannian setting (and in particular in the classical non-weighted one), that it is stable under measured Gromov-Hausdorff convergence, and that it implies various geometric and analytic inequalities relating metric and measure, in complete analogy with the smooth setting. It was subsequently also shown \cite{Ohta-CDforFinsler,Petrunin-CDforAlexandrov} that Finsler manifolds and Alexandrov spaces satisfy the Curvature-Dimension condition. 
Thus emerged an overwhelmingly convincing notion of Ricci curvature lower bound $K$ and dimension upper bound $N$ on metric-measure spaces, leading to a rich and fruitful theory exploring the geometry of such spaces by means of Optimal Transport.

\smallskip
However, one interesting setting in which the $\CD$ theory is not applicable (at least, not directly) is the sub-Riemannian one. It was first shown by Juillet \cite{Juillet-HeisenbergMCP} that the $d$-dimensional Heisenberg group $\H^d$, which is the simplest example of a non-trivial sub-Riemannian manifold, equipped with the Carnot--Carath\'eodory metric and left-invariant Lebesgue measure, does not satisfy the $\CD(K,N)$ condition for \emph{any} $K,N \in \R$. In \cite{Juillet-SubRiemannianNotCD}, Juillet extended this observation to completely general (strictly) sub-Riemannian manifolds (in which the rank of the distribution is nowhere maximal) equipped with an arbitrary smooth positive measure. On the other hand, Juillet showed in \cite{Juillet-HeisenbergMCP} that the Heisenberg group $\H^d$ (of topological dimension $n=2d+1$) does satisfy the property $\MCP(0,N)$ for $N = n+2$. The latter is a particular case of the Measure Contraction Property $\MCP(K,N)$,  introduced  independently by Ohta \cite{Ohta-MCP} and Sturm \cite{SturmCD2} as a weaker variant of the $\CD(K,N)$ condition. 
More general Carnot groups were subsequently shown to satisfy $\MCP(0,N)$ for appropriate $N$ by Rifford, Barilari and Rizzi \cite{Rifford-CarnotMCP, Rizzi-MCPonCarnot,BarilariRizzi-MCPonHTypeCarnot}. Additional examples of sub-Riemannian spaces verifying $\MCP$ (but not $\CD$) have been found in \cite{BadreddineRifford,BarilariRizzi-BE,BarilariRizzi-Interpolation,BGMR-HType,LeeLiZelenko}, such as generalized H-type groups, the Grushin plane, and various H-type foliations, Sasakian and $3$-Sasakian structures. 
\smallskip

In the past year, the study of $\MCP$ spaces has seen some increased activity, starting from the work of Cavalletti and Santarcangelo \cite{CavallettiSantarcangeloMCP} who obtained sharp isoperimetric inequalities, and continuing with the work of Han--Milman \cite{HanEMilman-MCP-Poincare} and Han \cite{Han-MCP-pPoincare} who obtained sharp Poincar\'e and $L^p$-Poincar\'e inequalities, respectively, for $\MCP(K,N)$ spaces whose diameter is upper-bounded by $D \in (0,\infty)$. While these results are sharp for the class of $\MCP$ spaces, as witnessed by equipping $(\R,|\cdot|)$ with an appropriate measure $\mm$, it remained unclear whether they provide good quantitative estimates for the above specific examples from the sub-Riemmanian setting, which certainly have
more structure than general $\MCP$ spaces. Moreover, the recent Jacobian interpolation inequalities \`a la Cordero-Erausquin--McCann--Schmuckenschl{\"a}ger \cite{CMSInventiones}, obtained by Balogh, Krist\'aly and Sipos  \cite{BKS-HeisenbergInterpolation,BKS-CorankOneCarnot} for the Heisenberg group and more general corank $1$ Carnot groups and by Barilari and Rizzi \cite{BarilariRizzi-Interpolation} in the ideal sub-Riemannian setting (see below), strongly suggest that more information can be extracted in these cases than by merely employing the $\MCP$ property.

\medskip
In this work, we introduce a new property we call Quasi Curvature-Dimension $\QCD(Q,K,N)$ ($Q \geq 1$), which constitutes a ``quasi-convex" relaxation of the $\CD(K,N)$ condition (the latter is recovered when the ``slack" parameter $Q$ is set to $1$), and serves as a bridge between the $\CD$ and $\MCP$ conditions. We draw our nomenclature from the theory of quasi-Banach spaces -- recall that a 1-homogeneous functional $\norm{\cdot}$ on a linear space $E$ is called a quasi-norm if $\exists Q \geq 1$ so that:
\[
\norm{(1-t) x_0 + t x_1} \leq Q \brac{ (1-t) \norm{x_0} + t \norm{x_1}} \;\;\; \forall x_0,x_1 \in E \;\;\; \forall t \in (0,1). 
\]
Roughly speaking, our main results in this work are as follows:
\begin{itemize}
\item In the above sub-Riemannian examples, and more generally, whenever an appropriate $(X,\d,\mm)$ satisfies a ``Jacobian" interpolation inequality and $\MCP(K,N)$ holds, then $\QCD(Q,K,N)$ holds as well with $Q = 2^{N-n}$, where $n$ denotes the topological dimension (in fact, modulo the results of \cite{BKS-HeisenbergInterpolation,BKS-CorankOneCarnot,BarilariRizzi-Interpolation}, this will be essentially trivial). 
This extends the well-known fact \cite[Corollary 5.5]{SturmCD2} that when $N=n$ (so that $Q = 2^{N-n} = 1$), the $\MCP(K,n)$ condition on unweighted Riemannian manifolds is equivalent to the $\QCD(1,K,n) = \CD(K,n)$ condition (i.e. to a lower bound $K$ on the Ricci curvature). 
\item Any property of $\CD(K,N)$ spaces which is amenable to \emph{localization} and in dimension one is stable under perturbations, also holds (up to constants depending only on $Q$) for $\QCD(Q,K,N)$ spaces (which are in addition essentially non-branching and also satisfy the $\MCP(K',N')$ condition for some $K' \in \R$ and $N' \in (1,\infty)$ -- see below for more details). For example, this applies to $L^p$-Poincar\'e inequalities, Sobolev and log-Sobolev inequalities, as well as isoperimetric inequalities. 
\end{itemize}
Consequently, we deduce that while (strictly) sub-Riemannian manifolds do not satisfy any type of $\CD$ condition, in the ideal and corank $1$ Carnot group settings, they satisfy (up to constants) most geometric and analytic properties as their $\CD$ counterparts. Moreover, the latter constants do not directly depend on the topological dimension $n$ nor the geodesic dimension $N$, but rather on their difference via the formula $Q = 2^{N-n}$, and so for any family of spaces for which $N-n$ is bounded above, they are dimension-independent (!); for example, $Q = 2^{N-n} = 4$ for all $d$-dimensional Heisenberg groups $\H^d$, regardless of $d$. 

\medskip

As a taste of the type of results one can obtain using these observations, we state the following consequence of our main Theorem \ref{thm:intro-main}. We refer to the next sections for precise definitions, and at this point only introduce the notation $\conv(\Omega)$ for the geodesic hull of a set $\Omega$, namely the union of all geodesics starting at $x \in \Omega$ and ending at $y \in \Omega$. Note that $\conv(\Omega)$ need not be geodesically convex, and that $\conv(B_r(x)) \subset B_{2r}(x)$ by the triangle inequality (where $B_r(x)$ denotes a geodesic ball around $x$ of radius $r > 0$) -- the latter is the prototypical example the reader should bear in mind below.

\begin{theorem} \label{thm:main-application}
Let $X$ be an ideal generalized H-type group of dimension $n$ and corank $k$, equipped with its Carnot--Carath\'eodory metric $\d$ and canonical left-invariant volume measure $\mm$. Then for all closed subsets $\Omega \subset X$ with $\diam(\Omega) \leq D < \infty$ and for any (locally) $\d$-Lipschitz function $f : (X,\d) \rightarrow \R$:
\begin{itemize}
\item The following Poincar\'e inequality holds:
\begin{equation} \label{eq:intro-Poincare}
\int_{\Omega} f \mm = 0  \;\; \Rightarrow \;\; \frac{\pi^2}{4^k D^2} \int_{\Omega} f^2 \mm \leq \int_{\conv(\Omega)} |\nabla_{X} f|^2 \mm .
\end{equation}
\item More generally, the following $L^p$-Poincar\'e inequality holds for any $p \in (1,\infty)$:
\[
\int_{\Omega} |f|^{p-2} f \mm = 0  \;\; \Rightarrow \;\; \frac{p-1}{4^k} \brac{\frac{2 \pi}{p \sin(\pi/p) D}}^p  \int_{\Omega} |f|^p \mm \leq \int_{\conv(\Omega)} |\nabla_{X} f|^p \mm .
\]
\item The following log-Sobolev inequality holds (for some universal numeric $C > 1$):
\[
\int_{\Omega} (f^2 - 1) \mm = 0 \;\; \Rightarrow \;\; \frac{\pi^2}{4^k 2 C D^2} \int_{\Omega} f^2 \log(f^2) \mm \leq \int_{\conv(\Omega)} |\nabla_{X} f|^2 \mm .
\]
\end{itemize}
In particular, this applies to all Heisenberg groups $\H^d$ with $k=1$ (independently of $d$). 
\end{theorem}

Analogous results hold for ideal Carnot groups, the (ideal) Grushin plane,  (ideal) Sasakian and $3$-Sasakian manifolds (under appropriate curvature lower bounds), 
 (ideal) H-type foliations with completely parallel torsion and non-negative horizontal sectional curvature, and general (possibly non-ideal) Carnot groups of corank $1$ -- see Section \ref{sec:results}.
To put these results into context, note that the Poincar\'e inequality (\ref{eq:intro-Poincare}) on the Heisenberg group $\H^d$ coincides up to a factor of $4$ with the celebrated Li--Yau / Zhong--Yang sharp spectral-gap estimate \cite{LiYauEigenvalues,ZhongYangImprovingLiYau}, which applies to geodesically convex subsets of $\CD(0,N)$ spaces \cite{BakryQianSharpSpectralGapEstimatesUnderCDD,KlartagLocalizationOnManifolds,CavallettiMondino-LocalizationApps}. Instead of assuming that $\Omega$ is geodesically convex, we use an arbitrary set $\Omega$ but take its geodesic hull $\conv(\Omega)$ on the energy side of the inequality -- this variant, originating in our previous work with B.~Han \cite{HanEMilman-MCP-Poincare}, is crucial in the sub-Riemannian setting, where non-trivial geodesically convex sets are known to be scarce; for instance, even for the simplest case of the Heisenberg group $\H^1$, it was shown in \cite{MontiRickly} that the smallest geodesically convex set containing three distinct points which do not lie on a common geodesic is $\H^1$ itself, implying in particular that there are no non-trivial geodesically convex balls in $\H^1$. 
Similarly, up to the factor of $4^k$, our estimates for the $L^p$-Poincar\'e inequality (spectral-gap of the $p$-Laplacian) and for the log-Sobolev inequality are known to be best possible on geodesically convex subsets of $\CD(0,N)$ spaces.

To the best of our knowledge, Theorem \ref{thm:main-application} entails the 
best known quantitative estimates for 
the $L^p$-Poincar\'e and log-Sobolev inequalities in the above mentioned sub-Riemannian setting,
and moreover, constitutes the first time that a dimension-independent quantitative estimate (not depending on $n$) has been established on the above spaces. 
 
While the validity of a local Poincar\'e inequality in the sub-Riemannian setting is well-known, starting from the work of D.~Jerison on vector fields satisfying H\"{o}rmander's condition \cite{Jerison-HormanderCondition} (see also \cite{DongLuSun} and the references therein), we are almost not aware of any explicit constants in any of these inequalities.
This includes the sub-elliptic Curvature-Dimension approach developed by Baudoin--Garofalo \cite{BaudoinGarofalo-SubRiemannianCD}, which was used by Baudoin--Bonnefont--Garofalo  in \cite[Theorem 4.2]{BBG-PoincareOnSubRiemannianViaCD} to obtain a local Poincar\'e inequality on various sub-Riemannian manifolds satisfying a \emph{non-negative} generalized Ricci curvature bound -- namely, for $\Omega = B_r(x)$ and with $B_{2r}(x)$ on the energy-side of the Poincar\'e inequality (in place of $\conv(B_r(x))$), these authors obtained a Poincar\'e constant of the form $\frac{C}{r^2}$ for all $r > 0$ and some non-explicit constant $C > 0$ depending on various additional curvature parameters and the underlying dimension. Note that by \cite{SobolevMetPoincare}, it is always possible to tighten (i.e. replace $B_{2r}(x)$ by $B_r(x)$ on the energy-side) a local $L^p$-Poincar\'e inequality on any length-space (see also \cite{Jerison-HormanderCondition}), but this would result in a further loss of explicit constants and dependence on the underlying dimension (via the doubling constant). The results of \cite{BBG-PoincareOnSubRiemannianViaCD} were extended to a possibly negative generalized Ricci bound by Kim \cite[Theorem 1.1]{Kim-GlobalPoincareForNegativelyCurvedSubRiemannian}. 
We remark that when the generalized Ricci curvature is \emph{strictly positive} in the sense of \cite{BaudoinGarofalo-SubRiemannianCD}, the global situation is simpler as the underlying measure is necessarily finite, and a global Poincar\'e as well as (a variant of) a log-Sobolev inequality, with explicit constants, were obtained by Baudoin--Bonnefont in \cite{BaudoinBonnefont-SGandLSForPositiveRicci}; however, it is not clear how to localize these estimates to geodesic balls. For gradient estimates on the heat-kernel on the Heisenberg group and its associated (global) Poincar\'e and log-Sobolev inequalities, see \cite{DriverMelcher-HeisenbergHeatKernel,Li-HeisenbergHeatKernel,  BBBC-HeisenbergHeatKernel, HebischZegarlinski-LogSobolevOnHTypeGroups, BonnefontChafaiHerry-HeisenbergHeatKernelLogSob}. 

The only prior \emph{explicit} estimates we are aware of for Poincar\'e and $L^p$-Poincar\'e inequalities on geodesic balls in the sub-Riemannian setting were just recently obtained in \cite{HanEMilman-MCP-Poincare} and \cite{Han-MCP-pPoincare}, respectively, but these only employed the $\MCP$ information, and thus are inevitably worse than the estimates of Theorem \ref{thm:main-application} by a factor \emph{exponential in the dimension $n$}.

\bigskip

We refer the reader to the next section for the definition of the $\QCD(Q,K,N)$ condition and statement of our main results. 
The rest of this work is organized as follows. In Section \ref{sec:prelim}, we recall some preliminaries from sub-Riemannian geometry and the theory of Optimal Transport. In Section \ref{sec:localization}, we prove a localization theorem for general interpolation coefficients. In Section \ref{sec:1D-QCD}, we study one-dimensional $\QCD$ densities. In Section \ref{sec:inqs}, we prove our main result on the equivalence (up to a factor of $Q$) between the best constants in various functional inequalities on $\QCD$ spaces and their $\CD$ counterparts. In Section \ref{sec:conclude} we provide some concluding remarks.

\bigskip
\noindent
\textbf{Acknowledgments.} I thank Fabrice Baudoin, Bangxian Han, Alexandru Krist\'aly and Luca Rizzi for their interest, comments and for providing additional references I was not aware of. 
I also thank the anonymous referee for their careful reading of the manuscript and detailed comments.

\section{Statement of the results}  \label{sec:results}

\subsection{Curvature via Interpolation}

The starting point of this work is the following interpolation inequality along $W_2$ geodesics.
It will be more convenient to state it using a dynamical plan $\nu$, namely a probability measure on $\geo(X,\d)$, the space of constant speed geodesics $\gamma$ parametrized on the unit-interval $[0,1]$. It is known that any $W_2$ geodesic $(\mu_t)_{t \in [0,1]}$ can be lifted to an optimal dynamical plan $\nu$ so that $(\e_t)_\sharp \nu  = \mu_t$ for all $t \in [0,1]$, where $\e_t(\gamma) = \gamma_t$ denotes the evaluation map. 

We will say that a metric-measure space $(X,\d,\mm)$ is \emph{Monge} if for any two probability measures with finite second moments $\mu_0,\mu_1 \in \P_2(X)$ with $\mu_0 \ll \mm$ and $\supp(\mu_1) \subset \supp(\mm)$, there exists a unique $W_2$ geodesic $[0,1] \ni t \mapsto \mu_t \in \P(X)$ connecting $\mu_0,\mu_1$, it is given by a map (there exists $S : X \rightarrow \geo(X,\d)$  so that $\nu = S_{\sharp} \mu_0$ is the associated optimal dynamical plan), and $\mu_t = (\e_t)_{\sharp} \nu \ll \mm$ for all $t \in [0,1)$. We refer to Section \ref{sec:prelim} for missing definitions and assertions, and only presently remark that in this work, a geodesic is always meant to mean minimizing geodesic, and that $\P_c(X)$ denotes the space of (Borel) probability measures on $X$ with bounded support.

Let $(\D,g)$ denote a sub-Riemannian structure on a smooth $n$-dimensional connected manifold $M$, and let $\d$ denote the associated Carnot--Carath\'eodory sub-Riemannian metric. Assume that $(M,\D,g)$ is ideal, namely that it admits no non-trivial abnormal geodesics and that $(M,\d)$ is complete. 
 Let $\mm$ denote a measure with smooth positive density with respect to some (any) volume measure on $M$. It follows from the work of McCann \cite{McCannOTOnManifolds} and Cordero-Erausquin--McCann--Schmuckenschl{\"a}ger \cite{CMSInventiones} in the complete Riemannian setting and of Figalli and Rifford \cite{FigalliRifford-subRiemannian} in the ideal sub-Riemannian one that $(M,\d,\mm)$ is a Monge space. 
  The following interpolation inequality was first established in the Riemannian setting by Cordero-Erausquin--McCann--Schmuckenschl{\"a}ger \cite{CMSInventiones}, and very recently extended to the ideal sub-Riemannian setting by Barilari and Rizzi \cite{BarilariRizzi-Interpolation}:
 
\begin{theorem}[Interpolation Inequality for ideal (sub-)Riemannian manifolds \cite{CMSInventiones,BarilariRizzi-Interpolation}] \label{thm:interpolation}
Let $(M,\d,\mm)$ denote an ideal sub-Riemannian manifold as above, let $\mu_0,\mu_1 \in \P_c(M)$ with $\mu_0,\mu_1 \ll \mm$, and let $\nu$ be the associated optimal dynamical plan. Denoting by $\rho_t := \frac{d\mu_t}{d\mm}$ the corresponding densities along the $W_2$ geodesic from $\mu_0$ to $\mu_1$, one has for any $t \in (0,1)$:
\begin{equation} \label{eq:interpolation-thm}
\rho^{-\frac{1}{n}}_t(\gamma_t) \geq \beta^{\frac{1}{n}}_{1-t}(\gamma_1,\gamma_0) \rho^{-\frac{1}{n}}_0(\gamma_0) + \beta^{\frac{1}{n}}_{t}(\gamma_0,\gamma_1) \rho^{-\frac{1}{n}}_1(\gamma_1) \;\;\; \text{for $\nu$-a.e. $\gamma \in \geo(M,\d)$}. 
\end{equation}
\end{theorem}
\noindent
Here $\beta_t(x,y)$ denotes the measure distortion coefficient from $x \in M$ to $y \in M$, defined as:
\begin{equation} \label{eq:intro-beta}
\beta_t(x,y) := \limsup_{r \rightarrow 0+} \frac{\mm(Z_t(\{x\},B_r(y)))}{\mm(B_r(y))} ,
\end{equation}
where  $Z_t(A,B)$ denotes the set of all $t$-midpoints between points $a \in A$ and $b \in B$ (if $A,B$ are Borel measurable, $Z_t(A,B)$ is analytic and hence $\mm$-measurable). 

On an $N$-dimensional Riemannian manifold whose Ricci curvature is bounded below by $K \in \R$, classical comparison theorems verify that  $\beta^{1/N}_t(x,y) \geq \tau^{(t)}_{K,N}(\d(x,y))$ (with equality on model spaces of constant sectional curvature $\frac{K}{N-1}$),
where:
\[ %\begin{equation} \label{eq:sigma}
 \tau^{(t)}_{K, N}(\theta) :=t^{\frac 1N} \Big ( \sigma^{(t)}_{K, N-1}(\theta)     \Big )^{1-\frac1N} ~,~ 
 \sigma^{(t)}_{K, N-1}   \big(\theta):=\left\{\begin{array}{llll}
+\infty &\text{if}~~ K \theta^2 \geq \pi^2 (N-1) ,\\
\frac{s_{K/(N-1)}(t \theta)}{s_{K/(N-1)}(\theta)} &\text{otherwise},\\
\end{array}
\right .
\] %\end{equation}
and:
\[ %\begin{equation} \label{eq:skappa}
s_\kappa(\theta):=\left\{\begin{array}{lll}
(1/\sqrt {\kappa}) \sin (\sqrt \kappa \theta), &\text{if}~~ \kappa>0,\\
\theta, &\text{if}~~\kappa=0,\\
(1/\sqrt {-\kappa}) \sinh (\sqrt {-\kappa} \theta), &\text{if} ~~\kappa<0.
\end{array}
\right.
\] %\end{equation}

The definitions of $\CD(K,N)$ given by Sturm \cite{SturmCD1,SturmCD2} and Lott--Villani \cite{LottVillaniGeneralizedRicci,LottVillani-WeakCurvature} may then be described in analogy to the above (sub-)Riemannian interpolation inequality. While their definitions are more involved (and slightly differ) on general metric-measure spaces $(X,\d,\mm)$ and for general $N \in [1,\infty]$, when $N \in (1,\infty)$ and on 
Monge spaces, the condition simplifies to requiring that for all $\mu_0,\mu_1 \in P_c(X)$ with $\mu_0,\mu_1 \ll \mm$ and for all $t \in (0,1)$:
\[
\rho^{-\frac{1}{N}}_t(\gamma_t) \geq \tau^{(1-t)}_{K,N}(\d(\gamma_0,\gamma_1)) \rho^{-\frac{1}{N}}_0(\gamma_0) + \tau^{(t)}_{K,N}(\d(\gamma_0,\gamma_1))  \rho^{-\frac{1}{N}}_1(\gamma_1) \;\;\; \text{for $\nu$-a.e. $\gamma \in \geo(X,\d)$}. 
\]
Similarly, the (weaker) $\MCP(K,N)$ condition on Monge spaces is defined by requiring that for all $\mu_0,\mu_1 \in P_c(X)$ with $\mu_0 \ll \mm$ and $\supp(\mu_1) \subset \supp \mm$, for all $t \in (0,1)$:
\[
\rho^{-\frac{1}{N}}_t(\gamma_t) \geq \tau^{(1-t)}_{K,N}(\d(\gamma_0,\gamma_1)) \rho^{-\frac{1}{N}}_0(\gamma_0) \;\;\; \text{for $\nu$-a.e. $\gamma \in \geo(X,\d)$}. 
\]
Equivalently, it is enough to check this for $\mu_0 = \frac{1}{\mm(B)} \mm\llcorner_B$ with bounded $B$ ($0 < \mm(B) < \infty$) and for $\mu_1 = \delta_{o}$ with $o \in \supp(\mm)$. In particular, it follows (since $\int_{\supp \mu_t} \rho_t = 1$) that: 
\[
\mm(Z_{1-t}(\{o\} , B)) = \mm(Z_t(B,\{o\})) \geq \mm(\supp \mu_t) \geq \tau^{(1-t)}_{K,N}(\Theta_{o,B})^N \mm(B) ,
\]
where:
\[
 \Theta_{o,B} := \begin{cases} \inf_{x \in B} \d(o,x) & K \geq 0 , \\ \sup_{x \in B} \d(o,x) & K  < 0 , \end{cases}  
\]
and we immediately conclude from (\ref{eq:intro-beta}) that on $\MCP(K,N)$ spaces:
\begin{equation} \label{eq:beta}
\beta_t(x,y) \geq \tau^{(t)}_{K,N}(\d(x,y))^N \;\;\; \forall t \in (0,1) .
\end{equation} 

It is presently not known whether the ideal assumption in Theorem \ref{thm:interpolation} can be removed (say, replaced with being complete and Monge). However, there is one non-ideal (yet still Monge) sub-Riemannian setting in which an analogous result has been established. The following was very recently shown by Balogh--Krist\'aly--Sipos \cite{BKS-CorankOneCarnot}:

\begin{theorem}[Interpolation Inequality for corank $1$ Carnot groups \cite{BKS-CorankOneCarnot}] \label{thm:interpolation2}
Let $M$ denote an $n$-dimensional corank $1$ Carnot group, endowed with its Carnot--Carath\'eodory sub-Riemannian metric $\d$ and left-invariant measure $\mm$. Then for any $\mu_0,\mu_1 \in \P_c(M)$ with $\mu_0,\mu_1 \ll \mm$, (\ref{eq:interpolation-thm}) holds. Furthermore:
\begin{equation} \label{eq:Rizzi}
\beta_{1-t}(\gamma_1,\gamma_0) \geq (1-t)^{n+2} \;\text{ and } \; \beta_{t}(\gamma_0 , \gamma_1) \geq t^{n+2} \;\;\; \text{for $\nu$-a.e. $\gamma \in \geo(M,\d)$}.
\end{equation}
\end{theorem}

In fact, (\ref{eq:Rizzi}) was previously shown by Rizzi \cite{Rizzi-MCPonCarnot}, thereby deducing that corank $1$ Carnot groups satisfy $\MCP(0,n+2)$. 
Note that a corank $1$ Carnot group $(M,\d,\mm)$ as above is indeed a Monge space, even though it may not be ideal -- see Subsection \ref{subsec:Monge} below.

\subsection{The Quasi Curvature-Dimension Condition} \label{subsec:QCD}

We are now ready to introduce the following definition and establish the subsequent proposition; we continue using the standard notation from the previous subsection.

\begin{definition}[Quasi Curvature-Dimension $\QCD(Q,K,N)$]
A Monge space $(X,\d,\mm)$ is said to satisfy the $\QCD(Q,K,N)$ condition, $Q \geq 1$, $K \in \R$, $N \in (1,\infty)$, if for all 
$\mu_0,\mu_1 \in P_c(X)$ with $\mu_0,\mu_1 \ll \mm$ 
and for all $t \in (0,1)$:
\begin{multline} \label{eq:def-QCD}
\rho^{-\frac{1}{N}}_t(\gamma_t) \geq \frac{1}{Q^{\frac{1}{N}}} \brac{\tau^{(1-t)}_{K,N}(\d(\gamma_0,\gamma_1)) \rho^{-\frac{1}{N}}_0(\gamma_0) + \tau^{(t)}_{K,N}(\d(\gamma_0,\gamma_1))  \rho^{-\frac{1}{N}}_1(\gamma_1)} \\
 \text{for $\nu$-a.e. $\gamma \in \geo(X,\d)$}. 
\end{multline}
\end{definition}

\medskip

\begin{proposition} \label{prop:MCP-QCD}
Let $(M,\D,g)$ denote an $n$-dimensional ideal sub-Riemannian manifold, let $\d$ denote the associated Carnot--Carath\'eodory sub-Riemannian metric, and let $\mm$ be a measure with smooth positive density on $M$. If $(M,\d,\mm)$ satisfies $\MCP(K,N)$ then it also satisfies $\QCD(Q,K,N)$ with $Q = 2^{N-n}$. 
\end{proposition}
\begin{proof}
By the preceding comments, we know that the $\MCP(K,N)$ condition implies (\ref{eq:beta}). Note that necessarily $N \geq n$, since otherwise this would mean that $\beta_t(x,y) \gg t^{n}$ as $t \rightarrow 0$, which is easily seen to be impossible (see e.g. \cite[Theorem 5]{BarilariRizzi-Interpolation}). Plugging this into the Interpolation Theorem \ref{thm:interpolation}, and applying Jensen's inequality:
\[
a,b \geq 0 ~,~ \alpha \in (0,1] \;\; \Rightarrow \;\; (a+b)^{\alpha} \geq 2^{\alpha-1} (a^{\alpha} + b^{\alpha}) 
\]
with $\alpha = \frac{n}{N} \in (0,1]$, we deduce that with the same notation used there, for all $t\in(0,1)$, for $\nu$-a.e. $\gamma \in \geo(X,\d)$:
\[
\rho^{-\frac{1}{N}}_t(\gamma_t) \geq \frac{1}{2^{\frac{N-n}{N}}} \brac{\tau^{(1-t)}_{K,N}(\d(\gamma_0,\gamma_1)) \rho^{-\frac{1}{N}}_0(\gamma_0) + \tau^{(t)}_{K,N}(\d(\gamma_0,\gamma_1))  \rho^{-\frac{1}{N}}_1(\gamma_1)} .
\]
\end{proof}

It was shown in 
\cite{Rifford-CarnotMCP,Rizzi-MCPonCarnot,BarilariRizzi-MCPonHTypeCarnot,LeeLiZelenko,BarilariRizzi-Interpolation,BarilariRizzi-BE,BGMR-HType} that general ideal Carnot groups, ideal generalized H-type groups and the Heisenberg group in particular, 
 the (ideal) Grushin plane, (ideal) Sasakian and $3$-Sasakian manifolds (under appropriate curvature lower bounds), 
and (ideal) H-type foliations with completely parallel torsion and non-negative horizontal sectional curvature,
when endowed with their canonical sub-Riemannian metric and volume measure, all satisfy $\MCP(0,N)$ for appropriate $N \in (1,\infty)$ (see these references and also \cite{BadreddineRifford} for additional non-ideal classes).
 It follows by Proposition \ref{prop:MCP-QCD} that in addition, they also satisfy $\QCD(Q,0,N)$ for appropriate $Q > 1$. 
We will only record the following particular instance which follows by combining Proposition \ref{prop:MCP-QCD} with \cite[Theorem 3]{BarilariRizzi-MCPonHTypeCarnot}
(cf. \cite[Subsection 7.2]{BarilariRizzi-Interpolation}). 

\begin{corollary} \label{cor:H-group-QCD}
Any ideal generalized H-type group $X$ of dimension $n$ and corank $k$, equipped with its Carnot--Carath\'eodory sub-Riemannian metric $\d$ and left-invariant  measure $\mm$, satisfies $\MCP(0 , n+2 k)$ and hence $\QCD(4^k, 0, n + 2k)$. In particular, this applies to all Heisenberg groups $\H^d$ with $n = 2 d +1$ and $k=1$. 
\end{corollary}

In the non-ideal setting, by invoking Theorem \ref{thm:interpolation2} instead of Theorem \ref{thm:interpolation} above, a completely identical argument for corank $1$ Carnot groups yields:

\begin{corollary} \label{cor:Step1-Carnot-QCD}
A corank $1$ Carnot group $X$ of dimension $n$, equipped with its Carnot--Carath\'eodory sub-Riemannian metric $\d$ and left-invariant  measure $\mm$, satisfies $\MCP(0 , n+2)$ and hence $\QCD(4, 0, n + 2)$.
\end{corollary}

See Subsection \ref{subsec:BM} for a discussion of the optimality of the constant $Q = 2^{N-n}$ in Proposition \ref{prop:MCP-QCD} (and in particular the constant $4^k$ in Corollary \ref{cor:H-group-QCD}) as well as the constant $Q=4$ in Corollary \ref{cor:Step1-Carnot-QCD}. 

\subsection{One-dimensional $\QCD$ spaces}

Up until now we have not really done anything of substance, besides applying Jensen's inequality and introducing the $\QCD$ definition, so we must now justify its usefulness. The latter stems
from the following one-dimensional observation. We denote by $\mathcal L^1$ the Lebesgue measure on $\R$. 

\begin{proposition} Let $h$ be a density on $\R$ which is continuous on its support. Then $(\R,\abs{\cdot},h \L^1)$ is a $\QCD(Q,K,N)$ space iff there exists a density $f$ on $\R$, continuous on its support, with:
\begin{equation} \label{eq:intro-stable}
h \leq f \leq Q h ,
\end{equation}
so that $(\R,\abs{\cdot},f \L^1)$ is a $\CD(K,N)$ space. 
\end{proposition}
This is proved in Corollary \ref{cor:QCD-density} and Proposition \ref{prop:QCD-CD}, by taking $f$ to be the ``$\CD(K,N)$ upper envelope" of $h$.
It is not too hard to realize that (\ref{eq:intro-stable}) is a genuinely one-dimensional property, and that an analogous necessary condition need not hold in higher dimensional settings without some dimension-dependence in the estimate; indeed, by Carath\'eodory's theorem, the convex hull in $\R^n$ can be realized by $n+1$ points but no less in general, and so any penalty incurred for ``quasi-concavity" between 2 points will be amplified as the dimension increases. Consequently, we need an apparatus for reducing the study of $\QCD$ spaces to the one-dimensional case. 

\subsection{General Localization Theorem}

We achieve this by extending the localization method -- a paradigm which reduces the task of establishing various analytic and geometric inequalities on an $n$-dimensional space to the one-dimensional setting -- to spaces satisfying general interpolation inequalities which include the $\QCD$ case. 

In the Euclidean setting, the localization method has its roots in the work of Payne and Weinberger \cite{PayneWeinberger} on the spectral-gap for convex domains in Euclidean space, and has been further developed by Gromov and V.~Milman \cite{Gromov-Milman} and Kannan, Lov\'asz and Simonovits \cite{KLS}. In a ground-breaking work \cite{KlartagLocalizationOnManifolds}, B.~Klartag reinterpreted the localization paradigm as a measure disintegration adapted to $L^1$-Optimal-Transport, and extended it to weighted Riemannian manifolds satisfying $\CD(K,N)$. In a subsequent breakthrough, Cavalletti and Mondino \cite{CavallettiMondino-Localization} (cf. \cite{CavallettiMondino-Laplacian}) have succeeded to extend this technique to Monge spaces satisfying $\CD(K,N)$ with $N < \infty$. The localization method is also available on Monge spaces satisfying $\MCP(K,N)$ with $N<\infty$ \cite{Cavalletti-MongeForRCD,CavallettiMondino-Laplacian}, starting from the work of Bianchini and Cavalletti in the non-branching setting \cite{BianchiniCavalletti-MongeProblem}. 

In Theorem \ref{thm:gen-loc}, we extend the localization method to Monge spaces for completely general interpolation coefficients. With our usual notation, it applies assuming that the Monge space is $\MCP(K',N')$ for some $K' \in \R$ and $N' \in (1,\infty)$, and that for a fixed $N \in (1,\infty)$ and coefficients $(0,1) \times \R_+ \ni (t,\theta) \rightarrow \sigma^{(t)}_i(\theta) \in [0,+\infty]$, $i=0,1$, which are continuous in each variable, the following interpolation property holds for all $t \in (0,1)$:
\[
\rho_t^{-\frac{1}{N}}(\gamma_t) \geq  \sigma^{(1-t)}_0(\d(\gamma_0,\gamma_1)) \rho_0^{-\frac{1}{N}}(\gamma_0) + \sigma^{(t)}_1(\d(\gamma_0,\gamma_1)) \rho_1^{-\frac{1}{N}}(\gamma_1) \;\;
\text{for $\nu$-a.e. $\gamma \in \geo(X,\d)$.}
\]
The proof is based on the proof of the localization theorem for $\CD(K,N)$ spaces by Cavalletti and Mondino \cite[Theorem 5.1]{CavallettiMondino-Localization}, with one crucial difference -- in \cite{CavallettiMondino-Localization}, the fact that the $\CD(K,N)$ condition on a one-dimensional geodesic enjoys the local-to-global property was extensively used, and so it was enough to establish it locally on geodesics participating in the localization. In contrast, the above condition employing general functions $\sigma_0,\sigma_1$ will typically \textbf{not} satisfy the local-to-global property even on a one-dimensional space (this is the case for $\QCD(Q,K,N)$ when $Q > 1$ and even $\MCP(K,N)$), and so we are required to directly obtain the global property on the geodesics. 

\subsection{Functional Inequalities on $\QCD$ spaces}

Combining all of the above ingredients, we are able to conclude that any property which is amenable to localization \emph{and} stable under perturbations as in (\ref{eq:intro-stable}), will be shared by $\QCD(Q,K,N)$ spaces together with their $\CD(K,N)$ counterparts, up to constants depending solely on $Q$. Fortunately, this includes a multitude of fundamental analytic and geometric properties; we will only demonstrate this for the $L^p$-Poincar\'e and log-Sobolev inequalities. We remark that the Poincar\'e inequality is sometimes also referred to as the ``Poincar\'e--Wirtinger" or ``Poincar\'e--Neumann" inequality in the literature. 

Given a metric-measure space $(X,\d,\mm)$, let $|{\nabla }_{X} f|: X \mapsto \R$ denote the local Lipschitz constant  of $f$, defined as
 \[
  |{\nabla }_{X} f|(x):=\limsup_{y \rightarrow x} \frac{|f(y)-f(x)|}{\d(y, x)} 
\]
(and $0$ if $x$ is an isolated point). Throughout this work, by ``locally Lipschitz function" we mean a locally $\d$-Lipschitz function. 
Assume that $\supp(\mm)$ is geodesically convex (any two points in $\supp(\mm)$ can be connected by a geodesic in $\supp(\mm)$). 
Given a subset $\Omega \subset \supp(\mm)$, recall that $\conv(\Omega)$ denotes its geodesic hull. 

\begin{itemize}
\item We denote by $\lambda_p[(X,\d,\mm),\Omega]$ the best constant $\lambda_p$ so that 
for any (locally) Lipschitz function $f : (X,\d) \rightarrow \R$, the following $L^p$-Poincar\'e inequality holds:
\[
\int_{\Omega} |f|^{p-2} f \mm = 0  \;\; \Rightarrow \;\; \lambda_{p} \int_{\Omega} |f|^p \mm \leq \int_{\conv(\Omega)} |\nabla_{X} f|^p \mm .
\]
\item We denote by $\lambda_{LS}[(X,\d,\mm),\Omega]$ the best constant $\lambda_{LS}$ so that 
for any (locally) Lipschitz function $f : (X,\d) \rightarrow \R$, the following log-Sobolev inequality holds:
\[
\int_{\Omega} (f^2 - 1) \mm = 0  \;\; \Rightarrow \;\; \frac{\lambda_{LS}}{2} \int_{\Omega} f^2 \log(f^2) \mm \leq \int_{\conv(\Omega)} |\nabla_{X} f|^2 \mm .
\]
\end{itemize}

The idea to use $\conv(\Omega)$ instead of $\Omega$ on the energy side of the functional inequalities above originated in our previous work with B.~Han \cite{HanEMilman-MCP-Poincare}, and enables us to get a meaningful inequality without imposing various extra conditions on $\Omega$. Indeed, if we were to replace $\conv(\Omega)$ by $\Omega$, the best constants above would clearly be $0$ for (say) disconnected $\Omega$, or even if $\Omega$ just contains arbitrarily small necks. One way to resolve this is to require that $\Omega$ be geodesically convex, but as already mentioned in the Introduction, this is too strong of an imposition on many spaces, especially in the sub-Riemannian setting, where geodesically convex subsets are known to be scarce.

Given a family $\mathscr{X}$ of metric measure spaces $(X,\d,\mm)$ so that $\supp(\mm)$ is geodesically convex and $D \in (0,\infty)$, we denote by $\Xi_{\mathscr{X},D}$ the collection of all $(\mathcal{X},\Omega)$ where $\mathcal{X} = (X,\d,\mm) \in \mathscr{X}$ and $\Omega$ is a closed subset of $\supp(\mm) \subset X$ with $\diam(\Omega) \leq D$. For any of our constants $\lambda_* \in \{ \lambda_p , \lambda_{LS} \}$, we set:
\begin{align*}
\lambda_*[\mathscr{X},D] & := \inf \{ \lambda_*[\mathcal{X},\Omega] \; ; (\mathcal{X},\Omega) \in \Xi_{\mathscr{X},D} \} , \\
\bar \lambda_*[\mathscr{X},D] & :=  \inf \{ \lambda_*[\mathcal{X},\supp(\mm_{\mathcal{X}})] \; ;  (\mathcal{X},\supp(\mm_{\mathcal{X}})) \in \Xi_{\mathscr{X},D} \} .
\end{align*}
Clearly $\lambda_*[\mathscr{X},D] \leq \bar \lambda_*[\mathscr{X},D]$. Note that the $\bar \lambda_*$ definition corresponds to simply integrating over $X$ (or equivalently $\supp(\mm)$) in both sides of the above inequalities; thus $\bar \lambda_*[\mathscr{X},D]$ is the best constant in these standard versions for all members of $\mathscr{X}$ so that $\diam(\supp(\mm)) \leq D$, whereas the $\lambda_*[\mathscr{X},D]$ variant gives us the added flexibility of considering arbitrary closed subsets of $\supp(\mm)$ of diameter at most $D$. 
In the one-dimensional setting, we additionally abbreviate for a density $h$ on $\R$ and a closed interval $I \subset \R$:
\[
\blambda_*[h,I] := \lambda_*[(\R,|\cdot|,h \L^1), I] .
\]

\begin{definition}[$\QCD_{reg}(Q,K,N)$, $\CD_{reg}(K,N)$ and $\CD_1(K,N)$]
Given $K \in \R$, $N \in (1,\infty)$ and $Q \geq 1$, we denote by $\QCD_{reg}(Q,K,N)$ the family of all Monge spaces $(X,\d,\mm)$ satisfying $\QCD(Q,K,N)$ and $\MCP(K',N')$ for some $K' \in \R$ and $N' \in (1,\infty)$; note that $\QCD_{reg}(1,K,N)$ coincides with the family $\CD_{reg}(K,N)$ of Monge spaces satisfying $\CD(K,N)$ (and hence $\MCP(K,N)$). We also denote by $\CD_1(K,N)$ the family of one-dimensional spaces $(\R,\abs{\cdot},h \L^1)$ satisfying $\CD(K,N)$. Note that:
\[
\CD_{1}(K,N) \subset \CD_{reg}(K,N)  \subset\QCD_{reg}(Q,K,N) .
\]
\end{definition}
It is known that $\supp(\mm)$ is geodesically convex on $\MCP(K,N)$ spaces, and hence for all of the above spaces. In the one-dimensional setting, it is not too hard to show that $\bar \lambda_*[\CD_1(K,N),D] = \lambda_*[\CD_1(K,N),D]$ (see Corollary \ref{cor:lambda-1D}). We can now state: 

\medskip

\begin{theorem} \label{thm:intro-main}
For all $K \in \R$, $N \in (1,\infty)$, $Q \geq 1$ and $D \in (0,\infty)$:
\begin{itemize}
\item $\bar \lambda_p[\CD_{1}(K,N),D] \geq \lambda_p[\QCD_{reg}(Q,K,N),D] \geq \frac{1}{Q} \bar \lambda_p[\CD_{1}(K,N),D]$ for all $p \in (1,\infty)$. 
\item $\bar \lambda_{LS}[\CD_{1}(K,N),D] \geq \lambda_{LS}[\QCD_{reg}(Q,K,N),D] \geq \frac{1}{Q} \bar \lambda_{LS}[\CD_{1}(K,N),D]$. 
\end{itemize}
\end{theorem}

The case $Q=1$ with the $\lambda_*$ middle term above replaced by (the a-priori larger) $\bar \lambda_*$  is not new, and was obtained by Cavalletti--Mondino \cite{CavallettiMondino-LocalizationApps} as an immediate corollary of their localization theorem for $\CD_{reg}(K,N)$ spaces; the possibility to extend this from $\bar \lambda_*$ to $\lambda_*$ as above was anticipated in our previous work \cite{HanEMilman-MCP-Poincare}, and is in itself new.
The case $Q > 1$ is the main novelty of Theorem \ref{thm:intro-main}, and constitutes the main result of this work. 

\smallskip
The constants $\bar \lambda_p[\CD_{1}(K,N),D]$ have been well-studied in the literature and completely determined:
\begin{equation} \label{eq:lambda_p}
\bar \lambda_p[\CD_{1}(K,N),D] = \blambda_p[ c^{N-1}_{K/(N-1)}(t) , [-D/2,D/2] ] ,
\end{equation}
where:
\[
c_\kappa(t):=\left\{\begin{array}{lll}
\cos (\sqrt \kappa t) 1_{[-\frac{\pi}{2} , \frac{\pi}{2}]}(\sqrt{\kappa} t) &\text{if}~~ \kappa>0,\\
1 &\text{if}~~\kappa=0,\\
 \cosh (\sqrt {-\kappa} t) &\text{if} ~~\kappa<0.
\end{array}
\right.
\]
This follows from the results of Bakry--Qian \cite{BakryQianSharpSpectralGapEstimatesUnderCDD} when $p=2$ (see also \cite{AndrewsClutterbuck-SharpSG,CalderonThesis}), 
and Matei \cite{Matei-pSG},
 Valtorta \cite{Valtorta-pSG},
Esposito--Nitsch--Trombetti \cite{ENT-pSGForLogConcaveDensity} 
 and Naber--Valtorta \cite{NaberValtorta-pSG} 
  for general $p \in (1,\infty)$ (see also \cite[Chapter 6]{CalderonThesis} and \cite{Xia-Finsler-pSG}); 
  in fact these authors directly showed in the weighted Riemannian setting that $\bar \lambda_p[\CD_{reg}(K,N),D] = \bar\lambda_p[\CD_{1}(K,N),D]$ prior to Klartag's extension of the localization method to the Riemannian setting. In particular (see \cite{Valtorta-pSG} and \cite{BakryQianSharpSpectralGapEstimatesUnderCDD}):
\begin{equation} \label{eq:LiYau}
\bar \lambda_p[\CD_{1}(0,N),D] = \blambda_p[1,[-D/2,D/2]] = (p-1) \brac{\frac{2 \pi}{p \sin(\pi/p) D}}^p  ~; 
\end{equation}
 \begin{equation} \label{eq:Lichnerowicz}
 K > 0 \;, \; D \geq \pi \sqrt{(N-1) / K} \;\; \Rightarrow \;\; \bar \lambda_2[\CD_{1}(K,N),D] = \frac{N}{N-1} K   .
 \end{equation}
 
Note that even in the simplest case of $p=2$ and $K \geq 0$, Theorem \ref{thm:intro-main} constitutes a \emph{sharp} and \emph{stable} extension of the celebrated Li--Yau / Zhong--Yang ($K=0$) and Lichnerowicz ($K > 0$) estimates \cite{LiYauEigenvalues,ZhongYangImprovingLiYau,LichnerowiczBook,EscobarLichnerowiczWithConvexBoundary,XiaLichnerowiczWithConvexBoundary} to the $\QCD(Q,K,N)$ setting -- indeed, setting $Q=1$ and applying Theorem \ref{thm:intro-main} to a geodesically convex $\Omega$ (so that $\conv(\Omega) = \Omega$) of diameter at most $D$, the latter sharp spectral-gap estimates are immediately recovered from (\ref{eq:LiYau}) and (\ref{eq:Lichnerowicz}), respectively. 
The same holds if we set $Q > 1$ and let $Q \rightarrow 1$. 

\smallskip

To the best of our knowledge, the model-densities on which the constants $\bar \lambda_{LS}[\CD_{1}(K,N),D]$ are attained have not been completely determined, although the natural conjecture is that the answer is the same as for $\bar \lambda_p$ in (\ref{eq:lambda_p}). Up to numeric constants $C,C' > 1$, this conjecture has been verified for $N=\infty$ by E.~Calderon \cite[Chapter 7]{CalderonThesis}, who showed that:
\[
\bar \lambda_{LS}[\CD_{1}(K,\infty),D] \geq \frac{1}{C} \blambda_{LS}[\exp(-K t^2/2) , [-D/2,D/2] ] \geq \frac{1}{C'} \begin{cases}  (-K)^{\frac{3}{2}} D e^{\frac{K D^2}{8}}  & K < -\frac{1}{D^2} \\ \max(K,\frac{1}{D^2}) & \text{otherwise}   \end{cases} .
\]
Note that $\bar \lambda_{LS}[\CD_{1}(K,N),D] = \frac{N}{N-1} K$ when $K > 0$ and $D \geq \pi \sqrt{(N-1)/K}$ by the Bakry--\'Emery estimate \cite{BakryEmery}. 
The case most interesting for us $K = 0$ is well-known to experts, and in particular:
\[
\bar \lambda_{LS}[\CD_{1}(0,N),D] \geq \frac{1}{C} \blambda_{LS}[1 , [-D/2,D/2] ]  = \frac{1}{C} \frac{\pi^2}{D^2} . 
\]

In conjunction with Corollary \ref{cor:H-group-QCD}, 
Theorem \ref{thm:intro-main} thus immediately yields Theorem \ref{thm:main-application} from the Introduction, which is the main application we have chosen to highlight in this work. Analogously, by invoking Theorem \ref{thm:intro-main} in conjunction with Proposition \ref{prop:MCP-QCD} (and recalling the subsequent comments), or alternatively in conjunction with Corollary \ref{cor:Step1-Carnot-QCD}, 
the $L^p$-Poincar\'e and log-Sobolev inequalities of Theorem \ref{thm:main-application} equally hold on the sub-Riemannian manifolds listed below (with their canonical sub-Riemannian metric and volume measure; the Sasakian and $3$-Sasakian manifolds below$^*$ require appropriate curvature lower bounds detailed in \cite{LeeLiZelenko,BarilariRizzi-Interpolation,BarilariRizzi-BE}; the H-type foliations below$^{**}$ are assumed to have completely parallel torsion and non-negative horizontal sectional curvature \cite{BGMR-HType}). Each of these spaces satisfies $\MCP(0,N)$, $\QCD(Q,0,N)$ and the inequalities of Theorem \ref{thm:main-application} with the values of $N$,$Q$ and $k$ given by the following table:
\vspace{-10pt}
\begin{center}
\begin{tabular}{|c|c | c|c | c | c |}
\hline
space & \makecell{necessarily\\ideal} & \makecell{ topological \\ dimension }&  $N$ & $Q$ & $k$ \\ \hline 
Grushin plane & yes & $n=2$ & $5$ \cite{BarilariRizzi-Interpolation} & $8$ & $3/2$   \\
\hline 
Sasakian manifolds$^*$ & yes & 
$n = 2d+1$ & $2d + 3$ \cite{LeeLiZelenko,BarilariRizzi-Interpolation}  & $4$ & $1$ \\
\hline 
$3$-Sasakian manifolds$^*$ & yes &
$n = 4d+3$ & $4d + 9$ \cite{BarilariRizzi-BE} & $64$ & $3$  \\
\hline 
ideal Carnot groups & yes &  $n$ & $N \in [n,\infty)$ \cite{Rifford-CarnotMCP} & $2^{N-n}$ & $\frac{N-n}{2}$  \\
\hline 
\makecell{H-type foliations$^{**}$\\of corank $k$} & yes & $n$ & $n + 2k$ \cite[$\S$3.7]{BGMR-HType} & $4^k$ & $k$  \\
\hline
\makecell{Carnot groups \\ of corank $1$} & no & $n$ & $n + 2$ \cite{Rizzi-MCPonCarnot, BKS-CorankOneCarnot} & $4$ & $1$  \\
\hline
\end{tabular}
\end{center}

%\vspace{-10pt}
%\begin{center}
%\begin{tabular}{|c|c | c|c | c | c |}
%\hline
%space & \makecell{necess-\\arily\\ideal} & \makecell{ topological \\ dimension }& \makecell{satisfies \\$\MCP(0,N)$} & \makecell{satisfies\\ $\QCD(Q,0,N)$} & \makecell{satisfies \\ inequalities \\ of Theorem\\\ref{thm:main-application}  with} \\ \hline 
%Grushin plane & yes & $n=2$ & $N=5$ \cite{BarilariRizzi-Interpolation} & $Q = 8$ & $k = 3/2$   \\
%\hline 
%Sasakian manifolds$^*$ & yes & 
%$n = 2d+1$ & $N = 2d + 3$ \cite{LeeLiZelenko,BarilariRizzi-Interpolation}  & $Q = 4$ & $k=1$ \\
%\hline 
%$3$-Sasakian manifolds$^*$ & yes &
%$n = 4d+3$ & $N = 4d + 9$ \cite{BarilariRizzi-BE} & $Q = 64$ & $k=3$  \\
%\hline 
%ideal Carnot groups & yes &  $n$ & $N \in [n,\infty)$ \cite{Rifford-CarnotMCP} & $Q = 2^{N-n}$ & $k = \frac{N-n}{2}$  \\
%\hline 
%\makecell{H-type foliations$^{**}$\\of corank $k$} & yes & $n$ & $N = n + 2k$ \cite[$\S$3.7]{BGMR-HType} & $Q = 4^k$ & $k$  \\
%\hline
%\makecell{Carnot groups \\ of corank $1$} & no & $n$ & $N = n + 2$ \cite{Rizzi-MCPonCarnot, BKS-CorankOneCarnot} & $Q = 4$ & $k=1$  \\
%\hline
%\end{tabular}
%\end{center}

\medskip

Specializing Theorem \ref{thm:main-application} to geodesic balls $\Omega = B_r(x)$, recall that $\conv(\Omega)$ appearing on the energy-side of the inequalities satisfies $\conv(\Omega) \subset B_{2r}(x)$, and so we obtain $L^p$-Poincar\'e and log-Sobolev inequalities on geodesic balls in non-tight form. 
As mentioned in the Introduction, it is always possible to tighten (i.e. replace $B_{2r}(x)$ by $B_r(x)$ on the energy-side) a local $L^p$-Poincar\'e inequality on any length-space, but this would result in a loss of explicit constants and dependence on the underlying dimension (via the doubling constant).

\section{Preliminaries} \label{sec:prelim}

\subsection{Sub-Riemannian Structures}

We refer to \cite{ABB-subRiemannianBook,FigalliRifford-subRiemannian,BarilariRizzi-Interpolation,BarilariRizzi-MCPonHTypeCarnot} and the references therein for more precise information and missing definitions pertaining to sub-Riemannian structures, as these will not be directly required in this work. Below we briefly describe some rudimentary notions. 

A sub-Riemannian structure on a smooth, connected $n$-dimensional manifold $M$ ($n \geq 3$), is defined by a set of $m$ global smooth vector fields $X_1,\ldots,X_m$, called a generating frame. The distribution $\D$ at the point $x \in M$ is defined as: \[
\D_x = \text{span} \{X_1(x),\ldots,X_m(x) \} \subset T_x M .
\]
The generating frame induces a natural inner product $g_x$ on $\D_x$. It is always assumed that the distribution satisfies H\"{o}rmander's bracket-generating condition (each tangent space $T_x M$ is spanned by the vector fields $\{X_i\}$ and their iterated Lie brackets evaluated at $x$). Being slightly imprecise, an absolutely continuous map $\xi : [0,1] \rightarrow M$ is called a horizontal curve if $\dot \xi(t) \in \D_x(\xi(t))$ for almost every $t \in [0,1]$. Its length is defined by:
\[
\ell(\xi) := \int_0^1 \sqrt{g(\dot \xi(t) , \dot \xi(t))} dt . 
\]
The Carnot--Carath\'eodory sub-Riemannian metric $d$ is then defined as:
\[
\d(x,y) := \inf \{ \ell(\xi) \; ; \; \xi(0) = x ~,~ \xi(1) = y ~,~ \xi \text{ is horizontal} \} . 
\]
By the Chow--Rashevskii theorem, the bracket-generating condition implies that $\d : M \times M \rightarrow \R$ is finite and continuous. We will always assume that $(M,\d)$ is complete, in which case the infimum above is always attained; a constant velocity horizontal curve realizing this infimum and parametrized on $[0,1]$ is called a geodesic. If in addition the sub-Riemannian structure $(\D,g)$ admits no abnormal geodesics between distinct points, it is called ideal; roughly speaking, this means that the differential of the end-point map $\xi \rightarrow \xi(1)$ on horizontal paths $\xi$ with fixed initial point $\xi(0)$, is non-singular for any geodesic $\gamma$ of positive length. It is known that complete fat sub-Riemannian structures are ideal, and that the ideal assumption is generic when the distribution $\D$ has constant rank at least $3$. 

In various places, we have emphasized how our results apply to generalized $H$-type groups. These are certain step $2$ Carnot groups, which include the Kaplan H-type groups and the Heisenberg groups, as well as all corank $1$ Carnot groups. A Carnot group of rank $r \geq 1$ and step $s \geq 1$ is a connected, simply connected nilpotent Lie group $G$, whose associated Lie algebra $\mathfrak{g}$ admits a stratification $\mathfrak{g} = \mathfrak{g}_1 \oplus \ldots \oplus \mathfrak{g}_s$ such that $\mathfrak{g}_1,\ldots,\mathfrak{g}_s$ are linear subspaces of $\mathfrak{g}$ satisfying $\mathfrak{g}_s \neq \{0\}$, $[\mathfrak{g}_1, \mathfrak{g}_i] = \mathfrak{g}_{i+1}$ for all $i=1,\ldots,s-1$, $[\mathfrak{g}_1, \mathfrak{g}_s] = \{0\}$, and the degree-one stratum $\mathfrak{g}_1$ has dimension $r$. A left-invariant sub-Riemannian structure is obtained by equipping $\mathfrak{g}_1$ with an inner product. Note that a corank $1$ Carnot group is necessarily of step $2$. 

The Heisenberg group $\H^d$ is an ideal Carnot group of corank $1$. Its elements are $(z_1,\ldots,z_d, t) \in \C^d \times \R \simeq \R^{2d+1}$, with the group structure given by:
\[
(z_1,\ldots,z_d,t) \cdot (z_1',\ldots,z_d',t') = (z_1 + z_1',\ldots,z_d + z_d' , t + t' + \frac{1}{2} \sum_{j=1}^d \text{Im}( \overline{z_j} z_j' )) . 
\]
Its bi-invariant Haar measure is just the Lebesgue measure $\L^{2d+1}$. Its sub-Riemannian structure is given by the global set of left-invariant generating fields:
\[
X_j = \partial_{x_j} - \frac{y_j}{2} \partial_{t} ~,~ Y_j = \partial_{y_j} + \frac{x_j}{2} \partial_{t} ,
\]
where $z_j = x_j + i y_j$. They satisfy the bracket relations $[X_j,Y_l] = \delta_{jl} Z$ and $[X_j , Z] = [Y_j,Z] = 0$, where $Z = \partial_t$.

\subsection{Optimal Transport}

Let $(X, \d)$ be a complete separable metric space endowed with a locally finite Borel measure $\mm$ -- 
such triplets $(X,\d,\mm)$ are called metric measure spaces. We refer to \cite{AmbrosioGigli-UsersGuide,AGS-Book,Gromov,VillaniTopicsInOptimalTransport, VillaniOldAndNew} for background on metric measure spaces in general, and the theory of optimal transport on such spaces in particular. 

We denote by $\geo(X,\d)$ the set of all closed directed constant-speed geodesics parametrized on the interval $[0,1]$. 
We regard $\geo(X,\d)$ as a subset of all Lipschitz maps $\text{Lip}([0,1], X)$ endowed with the uniform topology. 
Recall that $(X,\d)$ is called geodesic if for any $x,y \in X$ there exists $\gamma \in \geo(X,\d)$ with $\gamma_0 = x$ and $\gamma_1 = y$. 
Given a subset $A$ of a geodesic space $(X,\d)$, we denote by $\conv(A)$ the geodesic hull of $A$, namely:
\[
\conv(A) := \cup_{\{ \gamma \in \geo(X,\d) \; ; \; \gamma_0,\gamma_1 \in A \}}  \gamma \; ;
\]
note that $(\conv(A),\d)$ need not be a geodesic space itself. 

The space of all Borel probability measures on $(X,\d)$ is denoted by $\mathcal{P}(X)$. It is naturally equipped with its weak topology, in duality with bounded continuous functions $C_b(X)$ over $X$.  
The subspace of those measures having bounded support is denoted by $\P_c(X)$, and those with finite second moment is denoted by $\mathcal{P}_{2}(X)$. 
The weak topology on $\mathcal{P}_{2}(X)$ is metrized by the $L^{2}$-Wasserstein distance $W_{2}$, defined as follows for any $\mu_0,\mu_1 \in \mathcal{P}(X)$:
\begin{equation}\label{eq:Wdef}
  W_2^2(\mu_0,\mu_1) := \inf_{ \pi} \int_{X\times X} \d^2(x,y) \, \pi(dx , dy),
\end{equation}
where the infimum is taken over all $\pi \in \mathcal{P}(X \times X)$ having $\mu_0$ and $\mu_1$ as the first and the second marginals, respectively; such candidates $\pi$ are called transference plans. 
It is known that the infimum in (\ref{eq:Wdef}) is always attained for any $\mu_0,\mu_1 \in \mathcal{P}(X)$; when this minimum is finite, the collection of transference plans realizing it, called optimal transference plans between $\mu_0$ and $\mu_1$, is denoted by $\Opt(\mu_0,\mu_1)$.

When $\mu_0,\mu_1 \in \mathcal{P}_2(X)$, then necessarily $W_2(\mu_0,\mu_1) < \infty$. In this case, it is known that a transference plan $\pi$ is optimal iff it is supported on a $\d^2$-cyclically monotone set. A set $\Lambda \subset X \times X$ is said to be $c$-cyclically monotone if for any finite set of points $\{(x_{i},y_{i})\}_{i = 1,\ldots,N} \subset \Lambda$ it holds
$$
\sum_{i = 1}^{N} c(x_{i},y_{i}) \leq \sum_{i = 1}^{N} c(x_{i},y_{i+1}),
$$
with the convention that $y_{N+1} = y_{1}$.

As $(X,\d)$ is a complete and separable metric space then so is $(\mathcal{P}_2(X), W_2)$. 
Under these assumptions, it is known that $(X,\d)$ is geodesic if and only if $(\mathcal{P}_2(X), W_2)$ is geodesic. Let $\e_t$ denote the evaluation map: 
\[
  \e_{t} : \geo(X,\d) \ni \gamma\mapsto \gamma_t \in X.
\]
A measure $\nu \in  \mathcal{P}(\geo(X,\d))$ is called an optimal dynamical plan if $(\e_0,\e_1)_{\sharp} \nu$ is an optimal transference plan; it easily follows in that case that
$[0,1] \ni t \mapsto (\e_t)_\sharp \nu$ is a geodesic in $(\mathcal{P}_2(X), W_2)$. It is known that any geodesic $(\mu_t)_{t \in [0,1]}$ in $(\mathcal{P}_2(X), W_2)$ can be lifted to an optimal dynamical plan $\nu$ so that $(e_t)_\sharp \nu  = \mu_t$ for all $t \in [0,1]$ (c.f. \cite[Theorem 2.10]{AmbrosioGigli-UsersGuide}). We denote by $\OptGeo(\mu_{0},\mu_{1})$ the space of all optimal dynamical plans $\nu$ so that $(e_i)_\sharp \nu = \mu_i$, $i=0,1$. By the preceding remarks, it follows that for any closed $\Omega \subset X$ so that $(\Omega,d)$ is geodesic, $\OptGeo(\mu_{0},\mu_{1})$ is non-empty for all $\mu_0,\mu_1 \in \mathcal{P}_2(\Omega)$. 

\subsection{Monge Spaces} \label{subsec:Monge}

\begin{definition}[Monge Space]
A metric measure space $(X,\d,\mm)$ will be called a Monge space, if for any $\mu_0,\mu_1 \in \P_2(X)$ with $\mu_0 \ll \mm$ and $\supp(\mu_1) \subset \supp(\mm)$, the following holds:
\begin{itemize}
\item There exists a unique optimal dynamical plan $\nu \in \OptGeo(\mu_0,\mu_1)$, and hence a unique optimal transference plan $\pi \in \Opt(\mu_0,\mu_1)$; 
\item $\nu$ is induced by a map, namely, there exists  $S : X \to \geo(X,d)$ such that $\nu = S_{\sharp} \mu_{0}$; 
\item Denoting $\mu_t = (\e_t)_{\sharp} \nu$, we have $\mu_t \ll \mm$ for all $t \in [0,1)$. 
\end{itemize}
\end{definition}

It follows from the work of McCann \cite{McCannOTOnManifolds} and Cordero-Erausquin--McCann--Schmuckenschl{\"a}ger \cite{CMSInventiones}  that (smooth, connected) complete Riemannian manifolds $(M,g)$ equipped with their induced geodesic distance $\d$ and volume measure $\vol_g$ are Monge spaces
(strictly speaking, this was shown for $\mu_0,\mu_1 \in \P_c(X)$, but the extension from $\P_c(X)$ to $\P_2(X)$ is nowadays standard -- see e.g. \cite[Subsection 3.4]{FigalliRifford-subRiemannian}).
It was shown by Figalli and Rifford \cite[Sections 3,4]{FigalliRifford-subRiemannian} that very general (smooth, connected) complete sub-Riemannian manifolds $(M,\D,g)$ equipped with their volume measure are also Monge spaces; for instance, this holds for all ideal sub-Riemannian structures \cite[Theorem 5.9]{FigalliRifford-subRiemannian} (see also \cite[Theorem 39]{BarilariRizzi-Interpolation}). 

Another sub-Riemannian setting where the results of \cite{FigalliRifford-subRiemannian} ensure the first two properties above is for general (possibly non-ideal) step 2 Carnot groups (cf. \cite[Subsection 1.3]{BarilariRizzi-MCPonHTypeCarnot} or \cite[Subsection 2.5]{BKS-CorankOneCarnot}). Ensuring the third property requires additional justification. This has been established in the literature under additional assumptions, like being corank $1$ \cite[Proposition 2.4]{BKS-CorankOneCarnot} or generalized H-type \cite[Corollary 4]{BarilariRizzi-MCPonHTypeCarnot} -- the idea is to combine the known $\MCP$ information for these spaces with the fact that step 2 Carnot groups do not admit branching minimizing geodesics, and invoke (\ref{eq:ENB-Monge}) below. 
As pointed out to us by the referee, the same argument actually applies to general step 2 Carnot groups without any further assumptions, since these spaces satisfy $\MCP(0,N)$ for some $N > 1$ by the results of Badreddine--Rifford \cite[Theorem 4]{BadreddineRifford}. 

Clearly, the Monge property continues to hold when the volume measure is replaced by any measure $\mm$ having smooth positive density with respect to the former.

\subsection{Essentially Non-Branching Spaces}

\begin{definition}[Essentially Non-Branching]
A subset $G \subset \geo(X,\d)$ of geodesics is called non-branching if for any $\gamma^{1}, \gamma^{2} \in G$ the following holds:
$$
\exists t \in (0,1) \;\;\; \gamma^{1}_{s} = \gamma^{2}_{s} \;\;\; \forall s \in [0,t] 
\quad 
\Longrightarrow 
\quad 
\gamma^{1}_{s} = \gamma^{2}_{s} \;\;\; \forall s \in [0,1].
$$
$(X,\d)$ is called \emph{non-branching} if $\geo(X,\d)$ is non-branching. 
$(X,\d, \mm)$ is called \emph{essentially non-branching}  if for any $\mu_{0},\mu_{1} \ll \mm$ in $\mathcal{P}_{2}(X)$, any $\nu \in \OptGeo(\mu_{0},\mu_{1})$ is concentrated on a Borel non-branching subset $G\subset \geo(X,\d)$.
\end{definition}
\noindent Recall that a measure $\nu$ on a measurable space $(\Omega,\F)$ is said to be concentrated on $A \subset \Omega$ if $\exists B \subset A$ with $B \in \F$ so that $\nu(\Omega \setminus B) = 0$.

\smallskip
The above definition was introduced in \cite{RajalaSturm-StrongCD} by Rajala and Sturm, who showed that $\RCD(K,\infty)$ spaces are essentially non-branching. 
The restriction to essentially non-branching spaces is natural and facilitates 
avoiding pathological cases: as an example of possible pathological behaviour we mention the
failure of the local-to-global property of $\CD(K,N)$ within this class of spaces; in particular,
a heavily-branching metric measure space verifying a local version of $\CD(0, 4)$ which does not verify $\CD(K,N)$ for any
fixed $K \in \R$ and $N \in [1,\infty]$ was constructed by Rajala in \cite{RajalaFailureOfLocalToGlobal}, while the local-to-global
property of $\CD(K,N)$ has been recently verified in \cite{CavallettiEMilman-LocalToGlobal} for essentially non-branching metric measure spaces (with finite $\mm$).

\smallskip

It is easy to realize that a Monge space is necessarily essentially non-branching (e.g. \cite[Corollary 6.15]{CavallettiEMilman-LocalToGlobal}). Conversely, it was shown by Cavalletti and Mondino in \cite{CavallettiMondino-ENB-MCP} that an essentially non-branching space satisfying the Measure Contraction Property $\MCP(K,N)$ (for some $K \in \R$ and $N \in (1,\infty)$, defined next) is a Monge space: 
\begin{equation} \label{eq:ENB-Monge}
\text{essentially non-branching} + \MCP(K,N) \;\; \Rightarrow \;\; \text{Monge} \;\; \Rightarrow \;\; \text{essentially non-branching} . 
\end{equation}

\subsection{$\MCP(K,N)$}

The Measure Contraction Property $\MCP(K,N)$, $N \in (1,\infty)$, introduced by Ohta \cite{Ohta-MCP} and Sturm \cite{SturmCD2}, is a certain weak variant of the Curvature-Dimension condition $\CD(K,N)$. On general metric measure spaces the two definitions slightly differ, but on essentially non-branching (and hence Monge) spaces they coincide.
Recall the definition of the functions $\sigma^{(t)}_{K,N-1}$ and $\tau^{(t)}_{K,N}$ from Section \ref{sec:results}.

\begin{definition}[Measure Contraction Property $\MCP(K,N)$ on Monge spaces]
A Monge space $(X,\d,\mm)$ is said to satisfy $\MCP(K,N)$ if for any $\mu_0,\mu_1 \in \P_2(X)$, $\mu_0 \ll \mm$ and $\supp(\mu_1) \subset \supp(\mm)$,
writing $\mu_t = (\ee_t)_{\#} \nu = \rho_t \mm$ where $\nu$ is the unique element of $\OptGeo(\mu_0,\mu_1)$,
we have for all $t \in [0,1)$:
\begin{equation} \label{eq:MCP-def}
\rho_t^{-1/N}(\gamma_t) \geq  \tau_{K,N}^{(1-t)}(\d(\gamma_0,\gamma_1)) \rho_0^{-1/N}(\gamma_0)  \;\;\; \text{for $\nu$-a.e. $\gamma \in \geo(X,\d)$} .
\end{equation}
\end{definition}

\noindent
In fact, as follows from e.g. \cite[Proposition 9.1]{CavallettiEMilman-LocalToGlobal}, it is enough to test the above for:
\begin{equation} \label{eq:MCP-point}
 \mu_0 = \frac{1}{\mm(B)} \mm\llcorner_{B} \; \text{ with } \; 0 < \mm(B) < \infty ~,~  \mu_1 = \delta_{o} \text{ with } o \in \supp(\mm). 
 \end{equation}

Since some of our results are formulated on essentially non-branching spaces, we also mention for completeness the a-priori weaker (but by (\ref{eq:ENB-Monge}), equivalent) definition on the latter spaces (see \cite[Proposition 9.1]{CavallettiEMilman-LocalToGlobal}): for any $\mu_0,\mu_1$ as in (\ref{eq:MCP-point}), one should \emph{require} the existence of $\nu\in\OptGeo(\mu_0,\mu_1)$ so that  $\mu_t := (e_t)_{\#} \Pi \ll \mm$ for all $t \in [0,1)$, and so that writing $\mu_t = \rho_t \mm$, (\ref{eq:MCP-def}) holds for each $t \in [0,1)$.

It was shown in \cite{Ohta-MCP,SturmCD2} that the following (sharp) Bonnet-Myers diameter bound holds:
\[
{\diam(\supp \mm)} \leq D_{K,N}  := \left\{\begin{array}{lll} \frac{\pi}{\sqrt{K / (N-1)}} & \text{if}~~ K > 0 , \\+\infty & \text{otherwise} 
; \end{array}  \right . 
\]
we remark that while this is obvious from our present definition and the fact that $\tau_{K,N}(\theta) = +\infty$ if $\theta \geq D_{K,N}$, the above bound was shown in \cite{Ohta-MCP} under an a-priori weaker (but ultimately equivalent) definition of $\MCP(K,N)$ where the set $B$ above is assumed to be a subset of $B(o, D_{K,N})$ and in addition $(\supp \mm, \d)$ is a-priori assumed to be a length-space.  

\subsection{$\CD(K,N)$}

The Curvature-Dimension condition $\CD(K,N)$ has been defined on a general metric measure space independently 
in several seminal works by Sturm and Lott--Villani: the case $N=\infty$ and $K \in \Real$ was defined in \cite{SturmCD1} and \cite{LottVillaniGeneralizedRicci}, the case $N \in [1,\infty)$ in \cite{SturmCD2} for  $K \in \Real$ and in \cite{LottVillaniGeneralizedRicci} for $K=0$ (and subsequently for $K \in \Real$ in \cite{LottVillani-WeakCurvature}). 

In this work, we will only require the definition for Monge spaces with $N \in (1,\infty)$.
\begin{definition}[$\CD(K,N)$ for Monge Spaces] \label{def:CDKN-Monge}
A Monge space $(X,\d,\mm)$ is said to satisfy $\CD(K,N)$ if
for any $\mu_0,\mu_1 \in \P_2(X)$ with $\mu_0,\mu_1 \ll \mm$,  writing $\mu_t = (\ee_t)_{\#} \nu = \rho_t \mm$ where $\nu$ is the unique element of $\OptGeo(\mu_0,\mu_1)$,
we have for all $t \in [0,1]$:
\[
\rho_t^{-1/N}(\gamma_t) \geq  \tau_{K,N}^{(1-t)}(\d(\gamma_0,\gamma_1)) \rho_0^{-1/N}(\gamma_0) + \tau_{K,N}^{(t)}(\d(\gamma_0,\gamma_1)) \rho_1^{-1/N}(\gamma_1) \;\;\; \text{for $\nu$-a.e. $\gamma \in \geo(X,\d)$} .
\]
\end{definition}

When $N \in (1,\infty)$, it is known that if $(X,\d,\mm)$ satisfies $\CD(K,N)$ or $\MCP(K,N)$ then $(\supp(\mm),\d)$ is proper (every closed bounded set is compact) and geodesic; in addition, by approximating $\delta_{o}$ by $\mu_1^\eps = \mm(B(o,\eps))^{-1}\mm \llcorner_{B(o,\eps)}$, it is also known that the $\CD(K,N)$ condition implies the $\MCP(K,N)$ one (e.g. \cite[Section 6]{CavallettiEMilman-LocalToGlobal}).

\begin{remark}
Note that the definitions of $\MCP(K,N)$ and $\CD(K,N)$ given in this section employ $\mu_0,\mu_1 \in \P_2(X)$, whereas the ones given in Section \ref{sec:results} employed $\mu_0,\mu_1 \in \P_c(X)$. On Monge spaces so that $(\supp(\mm),\d)$ is proper (the proof of properness is valid for either variant), these two variants are completely equivalent -- see e.g. the proof of \cite[Proposition 9.1]{CavallettiEMilman-LocalToGlobal}. \end{remark}

\subsection{$\MCP(K,N)$ densities}

\begin{definition}[$\MCP(K,N)$ density]
A non-negative $h \in L^1_{loc}(\R,\mathcal L^1)$ is called an $\MCP(K,N)$ density if:
\[ h(tx_1+(1-t)x_0) \geq \sigma^{(1-t)}_{K, N-1} (|x_1-x_0|)^{N-1} h(x_0)
\] for all $x_0, x_1 \in \supp h$ and $t\in [0,1]$. 
\end{definition}

We use $\supp h$ throughout this work to denote $\supp(h \mathcal L^1)$, where recall, $\mathcal L^1$ denotes the Lebesgue measure on $\R$.
The following is well-known (see e.g. \cite[Lemma 4.1]{HanEMilman-MCP-Poincare}):

\begin{lemma} \label{lem:MCP-density-1} 
The one-dimensional metric-measure space $(\R, |\cdot|,  h \mathcal L^1)$ satisfies $\MCP(K,N)$ if and only if (up to modification on a null-set) $h$ is a $\MCP(K,N)$ density. 
\end{lemma}

\begin{lemma} \label{lem:MCP-density-2}
Let $h$ be an $\MCP(K,N)$ density. Then $\supp h \subset \R$ is a closed interval, $h$ is locally bounded above on $\supp h$, it is positive and locally Lipschitz on its interior $\intt \supp h$, and we may modify $h$ at the end points $\supp h \setminus \intt \supp h$ so that $h$ is continuous on $\supp h$. 
\end{lemma}
\begin{proof}
By definition of $\MCP(K,N)$ density, $\supp h$ is clearly convex, and is thus a closed interval. As follows from \cite[Lemmas A.8 and A.9]{CavallettiEMilman-LocalToGlobal} (which were stated for $\CD(K,N)$ densities of finite mass, but the proof only uses the defining property of $\MCP(K,N)$ densities and the local properties only require locally finite mass), $h$ is locally bounded above on $\supp h$, and is positive and locally Lipschitz on $\intt \supp h$. Lastly, since an $\MCP(K,N)$ density is clearly lower semi-continuous, we may modify the values of $h$ at the end points if necessary to ensure that $h$ is continuous on the entire $\supp h$.
\end{proof}
 
\subsection{Localization on $\MCP$ spaces}

Recall that given a measure space $(X,\mathscr{X},\mm)$, a set $A \subset X$ is called $\mm$-measurable if $A$ belongs to the completion of the $\sigma$-algebra $\mathscr{X}$, generated by adding to it all subsets of null $\mm$-sets; similarly, a function $f : (X,\mathscr{X},\mm) \rightarrow \R$ is called $\mm$-measurable if all of its sub-level sets are $\mm$-measurable. We denote by $\mathcal{M}(X,\mathscr{X})$ the collection of measures on $(X,\mathscr{X})$. 
We denote by $\Haus^1$ the one-dimensional Hausdorff measure on the underlying metric space. 

\begin{definition}[Disintegation on sets] \label{def:disintegration}
\label{defi:dis}
Let $(X,\mathscr{X},\mm)$ denote a measure space. 
Given any family $\{X_q\}_{q \in \Q}$ of subsets of $X$, a \emph{disintegration of $\mm$ on $\{X_q\}_{q \in \Q}$} is a measure-space structure 
$(\Q,\mathscr{Q},\qq)$ and a map
\[
\Q \ni q \longmapsto \mm_{q} \in \mathcal{M}(X,\mathscr{X})
\]
so that:
\begin{itemize}
\item For $\qq$-a.e. $q \in \Q$, $\mm_q$ is concentrated on $X_q$.
\item For all $B \in \mathscr{X}$, the map $q \mapsto \mm_{q}(B)$ is $\qq$-measurable.
\item For all $B \in \mathscr{X}$, $\mm(B) = \int_{\Q} \mm_{q}(B)\, \qq(dq)$; this is abbreviated by  $\mm = \int_{\Q} \mm_{q} \qq(dq)$.
\end{itemize}
\end{definition}

\begin{theorem}[Localization on $\MCP(K,N)$ spaces]\label{thm:localization-MCP}
Let $(X,\d,\mm)$ be an essentially non-branching metric measure space
satisfying the $\MCP(K,N)$ condition for some $K \in \R$ and $N \in (1, \infty)$. Let $g : X \rightarrow \R$ be $\mm$-integrable with $\int_X g \mm = 0$ and $\int_X|g(x)| \d(x,x_0) \mm(dx) < \infty$ for some (equivalently, all) $x_0 \in X$. Then there exists an $\mm$-measurable subset $\T \subset X$ and a family $\{X_{q}\}_{q \in \Q} \subset X$, such that: 
\begin{enumerate}
\item There exists a disintegration of $\mm\llcorner_{\T}$ on $\{X_{q}\}_{q \in \Q}$:
\[
\mm\llcorner_{\T} = \int_{\Q} \mm_{q} \, \qq(dq)  ~,~ \qq(\Q) = 1 . 
\]
\item For $\qq$-a.e. $q \in \Q$, $X_q$ is a closed geodesic in $(X,\d)$. 
\item For $\qq$-a.e. $q \in \Q$, $\mm_q$ is a Radon measure supported on $X_q$ with $\mm_q \ll  \Haus^1\llcorner_{X_q}$. 
\item For $\qq$-a.e. $q \in \Q$, the metric measure space $(X_{q}, \d,\mm_{q})$ verifies $\MCP(K,N)$.
\item For $\qq$-a.e. $q \in \Q$, $\int g \mm_q = 0$, and $g \equiv 0$  $\mm$-a.e. on $X \setminus \T$. 
\end{enumerate}
\end{theorem}

The localization paradigm on $\MCP(K,N)$ spaces has its roots in the work of Bianchini and Cavalletti in the non-branching setting (c.f. \cite[Theorem 9.5]{BianchiniCavalletti-MongeProblem}), and was extended to essentially non-branching $\MCP(K,N)$ spaces with $N < \infty$ and finite $\mm$ in \cite[Theorem 7.10 and Remark 9.2]{CavallettiEMilman-LocalToGlobal} (building upon \cite{Cavalletti-MongeForRCD}) and for general $\mm$ in \cite[Theorem 3.5]{CavallettiMondino-Laplacian}. The idea to use $L^1$-transport between the positive and negative parts $g_+ := \max(g,0)$ and $g_- := (-g)_+$ of the balanced function $g$ to ensure that it remains balanced along the localization is due to Klartag \cite{KlartagLocalizationOnManifolds} (see \cite{CavallettiMondino-Localization} for an adaptation to the metric measure space setting).

\begin{proof}[Proof of Theorem \ref{thm:localization-MCP}]
Simply combine \cite[Theorem 3.5]{CavallettiMondino-Laplacian} with the proof of \cite[Theorem 5.1]{CavallettiMondino-Localization}. Up to modification on a $\mm$-null-set, the set $\T$ is the transport set of the $1$-Lipschitz Kantorovich potential $u$ associated to the $L^1$-Optimal-Transport between $g_+ \mm$ and $g_- \mm$, which consists of geodesics $\{X_q\}$ on which the function $u$ is affine with slope $1$; for more details, see the proof of Theorem \ref{thm:gen-loc} below. 
 \end{proof}

\section{A general localization theorem} \label{sec:localization}

Our first observation in this work is the following:

\begin{theorem}[General Localization Theorem] \label{thm:gen-loc}

Let $(X,\d,\mm)$ be an essentially non-branching metric measure space
satisfying the $\MCP(K',N')$ condition for some $K'\in \R$ and $N' \in (1, \infty)$; in particular, the space is Monge by (\ref{eq:ENB-Monge}). 

Let $N \in (1,\infty)$, and let $(0,1) \times \R_+ \ni (t,\theta) \rightarrow \sigma^{(t)}_i(\theta) \in [0,+\infty]$, $i=0,1$, be continuous in each variable.
Assume that:
\begin{itemize}
\item
for all $\mu_0,\mu_1 \in \P_c(X)$ with $\mu_0,\mu_1 \ll \mm$, 
writing $\mu_t = (\ee_t)_{\#} \nu = \rho_t \mm$ where $\nu$ is the unique element of $\OptGeo(\mu_0,\mu_1)$,
we have for all $t \in (0,1)$:
\begin{multline} \label{eq:gen-loc-assumption}
\rho_t^{-\frac{1}{N}}(\gamma_t) \geq  (1-t)^{\frac{1}{N}} \sigma^{(1-t)}_0(\d(\gamma_0,\gamma_1))^{\frac{N-1}{N}} \rho_0^{-\frac{1}{N}}(\gamma_0) + t^{\frac{1}{N}} \sigma^{(t)}_1(\d(\gamma_0,\gamma_1))^{\frac{N-1}{N}} \rho_1^{-\frac{1}{N}}(\gamma_1) , \\
\text{for $\nu$-a.e. $\gamma \in \geo(X,\d)$. }
\end{multline}
\end{itemize}

 Let $g : X \rightarrow \R$ be $\mm$-integrable with $\int_X g \mm = 0$ and $\int_X|g(x)| \d(x,x_0) \mm(dx) < \infty$ for some (equivalently, all) $x_0 \in X$. 
 Then all the conclusions of Theorem \ref{thm:localization-MCP} hold, and in addition:
\begin{enumerate}
\setcounter{enumi}{5}
\item \label{item:loc-new} For $\qq$-a.e. $q \in \Q$,  $\mm_q = h_q \mathcal H^1\llcorner_{X_q}$ with continuous density $h_q : X_q \rightarrow \R_+$ satisfying: 
\[
h^{\frac{1}{N-1}}_q(x_t) \geq \sigma^{(1-t)}_0(\d(x_0,x_1)) h^{\frac{1}{N-1}}_q(x_0) + \sigma^{(t)}_1(\d(x_0,x_1)) h^{\frac{1}{N-1}}_q(x_1) \;\;\; \forall x_0,x_1 \in X_q \;\; \forall t \in (0,1) ,
\]
where $x_t$ denotes the unique point on $X_q$ so that $\d(x_t,x_0) = t \d(x_0,x_1)$ and $\d(x_t,x_1) = (1-t) \d(x_0,x_1)$. 
\end{enumerate}
\end{theorem}

\begin{remark}
To handle infinite values of $\sigma_i$, we use the convention that $\infty \cdot 0 = 0$.
\end{remark} 

\begin{remark}
The assumption that the space is $\MCP(K',N')$ may be relaxed, and is only included to guarantee some a-priori good properties of the space, like being Monge, being proper and having absolutely continuous conditional measures $\mm_q \ll \mathcal H^1\llcorner_{X_q}$ with continuous densities in the disintegration of $\mm \llcorner_{\T}$. For reasonable choices of $\sigma_0,\sigma_1$ this would in any case be guaranteed, but we avoid this extraneous generality, especially since we would like to apply the localization theorem in the $\QCD$ setting to functions for which $\sigma_0^{(1)},\sigma_1^{(1)} < 1$. 
\end{remark}

The proof below is based on the proof of the localization theorem for essentially non-branching $\CD(K,N)$ spaces by Cavalletti and Mondino \cite[Theorem 5.1]{CavallettiMondino-Localization}. However, as already mentioned in Section \ref{sec:results}, there is one crucial difference -- in \cite{CavallettiMondino-Localization}, the authors extensively used the fact that the $\CD(K,N)$ condition on a one-dimensional metric measure space enjoys the local-to-global property, and so it is enough to establish it locally on the geodesic $X_q$. Consequently, the authors only required the local $\CD_{loc}(K,N)$ condition to deduce their localization theorem. In contrast, the above condition employing general functions $\sigma_0,\sigma_1$ will typically \textbf{not} satisfy the local-to-global property even on a one-dimensional space (for example, this is the case for $\MCP(K,N)$ when $\sigma_1 = 0$ or for $\QCD(Q,K,N)$ when $Q > 1$), and so we are required to directly obtain the global property on $X_q$. This requires modifying the argument in several places and taking care of some additional technical points. 

\begin{proof}[Proof of Theorem \ref{thm:gen-loc}]
Without loss of generality, we may assume that $\supp(\mm) = X$, otherwise we restrict from $(X,\d,\mm)$ to $(\supp(\mm),\d,\mm)$ without altering any of the above properties of the space (see e.g. \cite[Section 6]{CavallettiEMilman-LocalToGlobal}). The $\MCP(K',N')$ assumption implies that $(X,\d)$ is proper and geodesic. 
Recall from Theorem \ref{thm:localization-MCP} the disintegration:
\begin{equation} \label{eq:disintegration}
\mm \llcorner_{\T} = \int_{\Q}  \mm_q \qq(dq) ,
\end{equation}
where for $\qq$-a.e. $q \in \Q$, $\mm_q$ is a Radon measure (in particular, finite on compact sets) supported on the closed geodesic $X_q$, $\mm_q \ll \Haus^1\llcorner_{X_q}$, and $(X_q,\d,\mm_q)$ satisfies $\MCP(K',N')$. Of course, we may identify $(X_q,\d,\mm_q)$ with $(I_q,|\cdot|,\bar h_q \L^1)$ for an appropriate closed interval $I_q \subset \R$ via a unit-speed parametrization of the geodesic $X_q$. 
Consequently, Lemmas \ref{lem:MCP-density-1} and \ref{lem:MCP-density-2} imply that for $\qq$-a.e. $q \in \Q$, we may write:
\[
\mm_q = h_q \Haus^1\llcorner_{X_q} ,
\]
with $h_q$ being an $\MCP(K',N')$ density which is continuous on $X_q$ and positive on its relative interior  $\relint X_q$. 

It remains to establish assertion (\ref{item:loc-new}) of Theorem \ref{thm:gen-loc}. 
To this end, let us recall from the work of Cavalletti and Mondino how the geodesics $X_q$ are constructed and how the disintegration (\ref{eq:disintegration}) is obtained (see \cite[Section 3]{CavallettiMondino-Localization}, \cite[Section 7]{CavallettiEMilman-LocalToGlobal} and \cite{Cavalletti-OTSurveyL1} for the case that $\mm$ is finite, and \cite[Section 3]{CavallettiMondino-Laplacian} for an adaptation to the case when $\mm$ is only assumed locally finite, and hence $\sigma$-finite by properness). Let $u$ denote the Kantorovich potential associated to the $L^1$-Optimal-Transport (corresponding to the cost $c(x,y) = \d(x,y)$) between $g_+ \mm$ and $g_- \mm$. 
Let $\Gamma := \{ (x,y) \in X \times X \; ;\; u(x) - u(y) = \d(x,y) \}$ and $\Gamma^{-1} := \{ (x,y) \in X \times X \; ; \; (y,x) \in \Gamma \}$. The \emph{transport relation} $R$ and the \emph{transport set} $\T$ are defined as:
\[
R := \Gamma \cup \Gamma^{-1} ~,~ \T := P_{1}(R \setminus \{ x = y \}) ,
\]
where $P_{i}$ is the projection onto the $i$-th component. Note that $R$ is closed, and it is easy to show that $\T$ is $\sigma$-compact. 
  The \emph{non-branched transport set} $\T^b$ is defined as $\T \setminus (A_+ \cup A_-)$, where $A_\pm$ denote the sets of forward and backward branching points, respectively (see \cite{CavallettiMondino-Localization}). The \emph{non-branched transport relation} is defined as $R^b := R \cap (\T^b \times \T^b)$. One can show that $A_{\pm}$ are $\sigma$-compact and hence $\T^b$ and $R^b$ are Borel. A crucial observation is that on Monge spaces of full support, $\mm(\T \setminus \T^b) = \mm(A_+ \cup A_-) = 0$. 
 
It turns out that $R^b$ is an equivalence relation over $\T^b$, and that for all $x \in \T^b$, $(R(x),\d)$ (where $R(x) := \{ y ; (x,y) \in R \}$) is isometric to a closed interval in $(\R,\abs{\cdot})$. Denote by $\Q$ the set of equivalence classes induced by $R^b$ over $\T^b$, and let $\QQ : \T^b \rightarrow \Q$ denote the quotient map. A disintegration theorem guarantees the existence of the 
 disintegration (\ref{eq:disintegration})  of $\mm\llcorner_{\T} = \mm\llcorner_{\T^b}$ strongly consistent with the partition 
of $\T^{b}$ given by the equivalence classes $\{R^b(q)\}_{q \in \Q}$ of $R^{b}$. 
The geodesics $\{X_q\}$ are obtained as the closure of each equivalence class in $\T^b$, and hence have disjoint relative interiors $\{\relint{X}_q\}$. Note that the function $u$ is affine on each $X_q$ with slope $1$.

As explained in \cite[Section 3]{CavallettiMondino-Localization} and \cite[Section 3]{CavallettiMondino-Laplacian}, up to modifying $\T^b$ and $\Q$ on $\mm$-null and $\qq$-null sets, respectively, the set $\Q$ can in fact be realized as a Borel subset of $\T^b$ so that (equipping $\Q$ with the trace $\sigma$-algebra) the quotient map $\QQ: \T^b \rightarrow \Q$ is Borel measurable and so that $\qq$ is a Borel probability measure on $\Q$. By inner regularity of Borel probability measures,
 it follows that, up to modification on a $\qq$-null set, $\Q$ is $\sigma$-compact; we write $\Q = \cup_{k=1}^\infty \Q^k$ with $\Q^k$ compact in $(X,\d)$.

The ray map $\r:  \Q \times \R \supset \text{Dom}(\r) \rightarrow \T^b$ is defined via:
\begin{align*}
\text{graph}(\r) & := \{ (q,t,x) \in \Q \times [0,\infty) \times \T^b \; ; \; (q,x) \in \Gamma ~,~ \d(q,x) = t \} \\
 & \cup  \{ (q,t,x) \in \Q \times (-\infty,0] \times \T^b \; ; \; (x,q) \in \Gamma ~,~ \d(x,q) = -t \} . 
\end{align*}
By definition $\text{Dom}(\r) := \r^{-1}(\T^b)$. 
It is known that $\r$ is a Borel map. After these preparations, we can finally commence the proof of assertion (6).

\medskip
Given $k$ and real parameters $a_0<a_1$ and $\eps_0,\eps_1 > 0$, denote:
\[
\Q^k_{a_0,a_1,\eps_0,\eps_1} := \set{ q \in \Q^k \; ; \; \text{$[\min(a_0-\eps_0,a_1 -\eps_1), \max(a_0+\eps_0,a_1+\eps_1)]$ is in the interior of $u(X_q)$} } .
\]
Note that:
\[
\Q^k_{a_0,a_1,\eps_0,\eps_1} = \Q^k \cap \bigcup_{n \geq 1} \left ( \begin{array}{ccc} P_1 \r^{-1}(u^{-1}(\min(a_0-\eps_0,a_1 -\eps_1) - 1/n))  \\ \cap \\  P_1 \r^{-1}(u^{-1}( \max(a_0+\eps_0,a_1+\eps_1) + 1/n)) \end{array} \right )   .
\]
Since $\Q^k$ is compact, since $u$ is Lipschitz and $\r$ is Borel, since the projection of a Borel set is analytic and hence universally measurable \cite[Section 4.3]{Srivastava-Book}, and since $\qq$ is a Borel measure, it follows that $\Q^k_{a_0,a_1,\eps_0,\eps_1}$ is $\qq$-measurable.

Let $a_0 < a_1$ and $\eps_0,\eps_1 > 0$ be such that $\qq(\Q^k_{a_0,a_1,\eps_0,\eps_1}) > 0$. Consider the measures:
\begin{equation} \label{eq:mui}
\mu_i := \frac{1}{\qq(\Q^k_{a_0,a_1,\eps_0,\eps_1})} \int_{\Q^k_{a_0,a_1,\eps_0,\eps_1}}  \frac{1}{2 \eps_i} \Haus^1\llcorner_{X_q} 1_{\set{ |u - a_i| \leq \eps_i}}  \qq(dq) \;\; , \;\; i=0,1 . 
\end{equation}
We postpone showing that $\mu_i$ are well-defined Borel measures on $(X,\d)$ (namely, the $\qq$-measurability of $q \mapsto \Haus^1\llcorner_{X_q}(B)$ given a Borel set $B \subset X$) to Lemma \ref{lem:measurability}. 
Since $[a_i- \eps_i , a_i + \eps_i] \subset u(X_q)$ for all $q \in \Q^k_{a_0,a_1,\eps_0,\eps_1}$, we see that $\mu_i$ are probability measures. Since $\Q^k$ is compact and $u$ is affine with slope $1$ on each $X_q$, it follows that $\mu_i$ are compactly supported. Moreover, we claim that $\mu_i \ll \mm$ with Radon--Nykodim derivative $\rho_i$ given by:
\begin{equation} \label{eq:rho}
\rho_i(x) := \frac{1}{\qq(\Q^k_{a_0,a_1,\eps_0,\eps_1})} \frac{1}{2 \eps_i} 1_{\set{ |u(x) - a_i| \leq \eps_i}} \frac{1}{h_q(x)} \text{ for } x \in \relint{X}_q ~,~ q \in \Q^k_{a_0,a_1,\eps_0,\eps_1}  \;\; , \;\; i=0,1 , 
\end{equation}
and $\rho_i(x) = 0$ otherwise. Indeed, this is a good definition for $\mm$-a.e. $x$, since the relative interiors of $X_q$ are disjoint (after perhaps removing a $\qq$-null set of $q$'s), and $\mm_q \ll \Haus^1\llcorner_{X_q}$ (and hence does not charge $X_q \setminus \relint X_q$) for $\qq$-a.e. $q$. 
Establishing the $\mm$-measurability of $\rho_i$ is postponed to Lemma \ref{lem:measurability}. 
It follows by (\ref{eq:disintegration}) that necessarily $\mu_i = \rho_i \mm$. 
Consider the map $T : X \rightarrow X$ which given $x \in \relint{X}_q$ with $q \in \Q^k_{a_0,a_1,\eps_0,\eps_1}$, produces the unique $T(x) \in X_q$ so that:
\begin{equation} \label{eq:T-def}
\frac{u(T(x)) - a_1}{\eps_1} = \frac{u(x) - a_0}{\eps_0} . 
\end{equation}
Let $G \subset X$ be the set on which $T$ is well-defined as described above. By the above arguments, 
$G$ has full $\mu_0$-measure (and hence may be assumed Borel), and on $G$ we have:
\[
T(x) = \r\brac{\QQ(x), u(\QQ(x)) - \eps_1 \frac{u(x) - a_0}{\eps_0} - a_1} ,
\]
so that $T$ is Borel measurable (as $\r$, $\QQ$ and $u$ are), and we have $T_{\sharp} \mu_0 = \mu_1$. Denote by $\pi \in \P(X \times X)$ the transference plan between $\mu_0$ and $\mu_1$ given by $(\Id \times T)_{\sharp} \mu_0$. 
We now use the following crucial observation due to Cavalletti \cite[Lemma 4.4]{Cavalletti-MongeForRCD} (cf. \cite[Lemma 4.1]{CavallettiMondino-Localization}), which connects the $L^1$-optimal-transport induced by $u$ with $L^2$-optimal-transport, and lies at the heart of the proof. 

\begin{lemma}
If $\Delta \subset X \times X$ is a set so that:
\[
(x_0,y_0) , (x_1,y_1) \in \Delta \;\; \Rightarrow \;\; (u(y_1) - u(y_0))(u(x_1) - u(x_0))  \geq 0 ,
\]
then $\Delta$ is $\d^2$-cyclically monotone. 
\end{lemma}

Note that the set $\Delta = \{(x,T(x)) ; x \in G\}$ satisfies the above property, since by (\ref{eq:T-def}):
\[
\frac{u(T(x_1)) - u(T(x_0))}{\eps_1} = \frac{u(x_1) - u(x_0)}{\eps_0} \;\;\; \forall x_0,x_1 \in G .  
\]
It follows that $\Delta$ is $\d^2$-cyclically monotone, and as $\pi$ is concentrated on $\Delta$, we deduce that $\pi$ is the (unique) optimal transference plan between $\mu_0$ and $\mu_1$.

Denoting by $\gamma_T(x)$ the geodesic from $x$ to $T(x)$ in $X_q$ (for $x \in G$), it follows that $\nu := (\gamma_T)_{\sharp} \mu_0$ is the (unique) optimal dynamical plan between $\mu_0$ and $\mu_1$. Setting $\mu_t = (\ee_t)_{\sharp} \nu$, we clearly have for all $t \in [0,1]$ that:
\[
\mu_t := \frac{1}{\qq(\Q^k_{a_0,a_1,\eps_0,\eps_1})} \int_{\Q^k_{a_0,a_1,\eps_0,\eps_1}}  \frac{1}{2 \eps_t} \Haus^1\llcorner_{X_q} 1_{\set{ |u - a_t| \leq \eps_t}}  \qq(dq)  ,
\]
where $a_t := (1-t) a_0 + t a_1$ and $\eps_t := (1-t) \eps_0 + t \eps_1$. Writing $\mu_t = \rho_t \mm$, we deduce from (\ref{eq:disintegration}) as before the following representation for the densities:
\[
\rho_t(x) = \frac{1}{\qq(\Q^k_{a_0,a_1,\eps_0,\eps_1})} \frac{1}{2 \eps_t} 1_{\set{ |u(x) - a_t| \leq \eps_t}} \frac{1}{h_q(x)} ,
\]
for $x \in \relint{X}_q$ and $\qq$-a.e. $q \in \Q^k_{a_0,a_1,\eps_0,\eps_1}$. 

For notational convenience, given a closed geodesic $X_q$, we identify it with the closure $L_q$ of the interval $(\inf  u(X_q) , \sup u(X_q)) \subset \R$ (by mapping $x \in X_q$ to the unique $s \in L_q$ so that $u(x) = s$). 
Applying our assumption (\ref{eq:gen-loc-assumption}), it follows that given $t \in (0,1)$, for $\qq$-a.e. $q \in \Q^k_{a_0,a_1,\eps_0,\eps_1}$, and for $\Haus^1$-a.e. $s_0 \in [a_0 - \eps_0 , a_0 + \eps_0]$, we have:
\[
\eps_t^{\frac{1}{N}} h^{\frac{1}{N}}_q(s_t) \geq  (1-t)^{\frac{1}{N}} \sigma^{(1-t)}_0(s_1 - s_0)^{\frac{N-1}{N}} \eps_0^{\frac{1}{N}} h^{\frac{1}{N}}_q(s_0) + t^{\frac{1}{N}} \sigma^{(t)}_1(s_1 - s_0)^{\frac{N-1}{N}} \eps_1^{\frac{1}{N}} h^{\frac{1}{N}}_q(s_1) ,
\]
where $s_t = (1-t) s_0 + t s_1$, and $s_1$ is given by:
\[
\frac{s_0 - a_0}{\eps_0} = \frac{s_1 - a_1}{\eps_1} . 
\]
Since $\sigma^{(1-t)}_0$ and $\sigma^{(t)}_1$ are assumed continuous, and since $h_q$ is continuous and positive 
on $\relint{L}_q$ for  $\qq$-a.e. $q$, the above actually holds for \emph{all} $s_0 \in [a_0 - \eps_0,a_0 + \eps]$ (and in particular, for $s_0 = a_0$), for $\qq$-a.e. $q \in \Q^k_{a_0,a_1,\eps_0,\eps_1}$. Namely, given $t \in (0,1)$, for any $k$, $a_0 < a_1$ and $\eps_0, \eps_1 > 0$, we have:
\begin{equation} \label{eq:gen-loc-key}
\eps_t^{\frac{1}{N}} h^{\frac{1}{N}}_q(a_t) \geq  (1-t)^{\frac{1}{N}} \sigma^{(1-t)}_0(a_1 - a_0)^{\frac{N-1}{N}} \eps_0^{\frac{1}{N}} h^{\frac{1}{N}}_q(a_0) + t^{\frac{1}{N}} \sigma^{(t)}_1(a_1 - a_0)^{\frac{N-1}{N}} \eps_1^{\frac{1}{N}} h^{\frac{1}{N}}_q(a_1) ,
\end{equation}
for $\qq$-a.e. $q \in \Q^k_{a_0,a_1 ,\eps_0,\eps_1}$. Enumerating over $k$ and all rational values of $a_0 < a_1$ and $\eps_0,\eps_1 > 0$, and using the continuity of $\sigma^{(1-t)}_0$ and $\sigma^{(t)}_1$ and also of $h_q$ on $\relint{L}_q$, it follows that given $t \in (0,1)$, there exists a single $\qq$-null set $\mathcal{N}_t$, so that for all $q \in \Q \setminus \mathcal{N}_t$, (\ref{eq:gen-loc-key}) holds for all $a_0 < a_1$ in $\relint{L}_q$ and $\eps_0,\eps_1 > 0$ small enough. Optimizing on the choice of $\eps_i>0$, we set:
\begin{align*}
\eps_0 & :=  \frac{\delta}{1-t} \cdot \frac{ \sigma^{(1-t)}_0(a_1 - a_0) h^{\frac{1}{N-1}}_q(a_0)}{ \sigma^{(1-t)}_0(a_1 - a_0) h^{\frac{1}{N-1}}_q(a_0) + \sigma^{(t)}_1(a_1 - a_0) h^{\frac{1}{N-1}}_q(a_1)} , \\
\eps_1 & := \frac{\delta}{t} \cdot \frac{ \sigma^{(t)}_1(a_1 - a_0) h^{\frac{1}{N-1}}_q(a_1)}{ \sigma^{(1-t)}_0(a_1 - a_0) h^{\frac{1}{N-1}}_q(a_0) + \sigma^{(t)}_1(a_1 - a_0) h^{\frac{1}{N-1}}_q(a_1)} ,
\end{align*}
for some small enough $\delta > 0$, and thus deduce from (\ref{eq:gen-loc-key}) that given $t \in (0,1)$, for all $q \in \Q \setminus \mathcal{N}_t$:
\begin{equation} \label{eq:gen-loc-key2}
h^{\frac{1}{N-1}}_q(a_t) \geq \sigma^{(1-t)}_0(a_1 - a_0) h^{\frac{1}{N-1}}_q(a_0) + \sigma^{(t)}_1(a_1 -a_0) h^{\frac{1}{N-1}}_q(a_1) \;\;\; \forall a_0,a_1 \in \relint{L}_q .
\end{equation}
In fact, since $h_q$ was modified to be continuous on the entire $L_q$, the above holds for all $a_0,a_1 \in L_q$, if we interpret $\infty \cdot 0$ as $0$ (recall that $\sigma_i$ are allowed to be infinite). 
It remains to apply this to all rational $t \in (0,1)$, and by invoking the continuity of $(0,1) \ni t \mapsto \sigma^{(t)}_i(\theta)$ and of $h_q$, we deduce that for $\qq$-a.e. $q \in \Q$, (\ref{eq:gen-loc-key2}) holds for all $a_0,a_1 \in L_q$ and $t \in (0,1)$. This concludes the proof. 
\end{proof}

It remains to address a couple of measurability issues which arose during the proof above; we continue using the same notation as there (see also an alternative argument in Remark \ref{rem:uniqueness} below). 

\begin{lemma} \label{lem:measurability}
\hfill
\begin{enumerate}
\item For any Borel set $B \subset X$, $\Q \ni q \mapsto \Haus^1\llcorner_{X_q}(B)$ is $\qq$-measurable. 
\item The map $\text{Dom}(\r) \ni (q,t) \mapsto h_q(\r(q,t))$ is $\qq \otimes \L^1$-measurable. 
\item The densities $\rho_i$ defined in (\ref{eq:rho}) are $\mm$-measurable. 
\end{enumerate}
\end{lemma}
\begin{proof}
\begin{enumerate}
\item Note that for $\qq$-a.e. $q \in \Q$, we have:
\[
\Haus^1\llcorner_{X_q}(B) = \Haus^1\llcorner_{\relint X_q}(B) = \int_{\text{Dom}(\r(q,\cdot))} 1_B(\r(q,t)) \L^1(dt) .
\]
Since $\text{Dom}(\r) \ni (q,t) \mapsto 1_B(\r(q,t))$ is a Borel function (as $\r$ and $B$ are Borel), and since $\text{Dom}(\r)$ is Borel, the first assertion follows.

\item It will be convenient to extend the definition of $h_q(r(q,t))$ to the entire $\Q \times \R$ by setting $H(q,t) := h_q(r(q,t)) 1_{\text{Dom}(\r)}(q,t)$. 
Given a compact interval $I \subset \R$, note that for $\qq$-a.e. $q$:
\[
\int_{I} H(q,\tau) \L^1(d\tau) = \mm_q(B_I) ~,~ B_I := \set{x \in \T^b \; ; \; u(\QQ(x)) - u(x) \in I } .
\]
Since $\QQ, u, \T^b, \text{Dom}(\r)$ are Borel, it follows that $B_I$ is Borel as well, as so by the measurability property of the disintegration (\ref{eq:disintegration}) we deduce that:
\begin{equation} \label{eq:meas1}
\Q \ni q \mapsto \int_{I} H(q,\tau) \L^1(d\tau) \text{ is $\qq$-measurable}. 
\end{equation}
 Note that for $\qq$-a.e. $q$, $\supp H(q,\cdot)$ coincides with the closure of $\text{Dom}(\r)(q,\cdot)$. 
  Since in addition, $\tau \mapsto H(q,\tau)$ is continuous on its support for $\qq$-a.e. $q$, by applying (\ref{eq:meas1}) to $I_\eps = [t-\eps,t+\eps]$ for a fixed $t \in \R$ and $\eps > 0$, dividing by $\L^1(I_\eps \cap \text{Dom}(\r)(q,\cdot))$ and taking the limit as $\eps \rightarrow 0$ (assuming the denominator is positive for all $\eps > 0$), it follows that for all $t \in \R$, the function:
\[
\Q \ni q \mapsto H(q,t)  \text{ is $\qq$-measurable}
\]
(we have used the standard fact that the pointwise limit of measurable functions is measurable, e.g. \cite[Proposition 3.1.27]{Srivastava-Book}). 

Now given $s > 0$, note that that continuity of $\tau \mapsto H(q,\tau)$ on its support for $\qq$-a.e. $q$ implies that up to a $\qq \otimes \L^1$ null-set:
\begin{multline*}
\{(q,t) \in \Q \times \R \; ; \; H(q,t) \geq s\} =\\
 \bigcap_{n \in \mathbb{N} , n > 1/s} \;\; \bigcup_{\tau \in \mathbb{Q}} \set{ q \in \Q \; ; \; H(q,\tau) \geq s-\frac{1}{n} } \times \set{ t \in \R \; ; \; |t - \tau| \leq \frac{1}{n} } 
\end{multline*}
(compare with the proof of \cite[Theorem 3.1.30]{Srivastava-Book}).
Since each of the product sets on the right-hand side is $\qq \otimes \L^1$ measurable, it follows that so is the left-hand side, concluding the proof of the second assertion. 

\item Note that the disintegration formula (\ref{eq:disintegration}) and the fact that $\mm_q \ll \Haus^1\llcorner_{X_q}$ for $\qq$-a.e. $q \in \Q$ together imply that if $D \subset \text{Dom}(\r)$ is such that $\qq \otimes \L^1 (D) = 0$ then:
\[
\mm(\r(D)) = \int \mm_q(\r(D)) \qq(dq) = \int_{D}  h_q(\r(q,t)) \qq\otimes\L^1(dq \; dt) = 0 .
\]
In particular, we see that $\cup_{q \in \Q} (X_q \setminus \relint X_q)$ and $\QQ^{-1}(Q_0)$ for any $\qq$-null set $Q_0$ are ($\mm$-measurable) $\mm$-null sets. It follows that $\mm$-a.e. we have:
\[
\rho_i(x) = \frac{1}{\qq(\Q^k_{a_0,a_1,\eps_0,\eps_1})} \frac{1}{2 \eps_i} 1_{\set{ |u(x) - a_i| \leq \eps_i}} \frac{1}{h_{\QQ(x)}(x)} 1_{\QQ^{-1}(\B^k_{B,a_0,a_1,\eps_0,\eps_1})}(x) , 
\]
where $\B^k_{a_0,a_1,\eps_0,\eps_1} \subset \Q$ is a Borel set which coincides with $\Q^k_{a_0,a_1,\eps_0,\eps_1}$ up to a $\qq$-null set. 
Since $\QQ : \T^b \rightarrow \Q$ is Borel and $u$ is Lipschitz, this reduces the task of establishing that $\rho_i$ is $\mm$-measurable to showing that $h_{\QQ(x)}(x)$ is $\mm$-measurable. Given $s  > 0$, the second assertion of the Lemma ensures that $\{(q,t) \in \Q \times \R \; ; \; h_q(\r(q,t)) \geq s \}$ is $\qq \otimes \L^1$ measurable, and hence may be written as $D_0 \triangle D$ where $D_0$ is a $\qq \otimes \L^1$-null set, and $D$ is a Borel subset of $\Q \times \R$. Since $\mm(\r(D_0)) = 0$, it follows that up to an $\mm$-null set:
\[
\set{ x \in \T^b \; ; \; h_{Q(x)}(x) \geq s } = \r \set{ (q,t) \in \Q \times \R \; ; \; h_q(\r(q,t)) \geq s } = \r(D)  .
\]
Since $D$ and $\r$ are Borel, $\r(D)$ is analytic (see \cite[Theorem 4.5.2]{Srivastava-Book}) 
and hence $\mm$-measurable \cite[Section 4.3]{Srivastava-Book}, thereby concluding the proof of the third assertion. 
\end{enumerate}
\end{proof}

\begin{remark} \label{rem:uniqueness}
Using essential uniqueness of disintegration, it is possible to avoid establishing the last two assertions of Lemma \ref{lem:measurability} directly, and argue in the proof of Theorem \ref{thm:gen-loc} as follows. First, note that $\mu_i \ll \mm$, since if $\mm(B) = 0$ then by the disintegration (\ref{eq:disintegration}) if follows that $\mm_q(B) = 0$ for $\qq$-a.e. $q$, and since $\mm_q$ and $\Haus^1\llcorner_{X_q}$ are mutually absolutely continuous for $\qq$-a.e. $q$, it follows that $\mu_i(B) = 0$ directly from the definition (\ref{eq:mui}). Using the disintegration (\ref{eq:disintegration}) again, we write:
\begin{equation} \label{eq:mui2}
 \mu_i = \frac{d\mu_i}{d\mm} \mm = \int_{\Q} \frac{d\mu_i}{d\mm} \mm_q \qq(dq) = \int_{\Q} \frac{d\mu_i}{d\mm} h_q \Haus^1\llcorner_{X_q} \qq(dq) .
\end{equation}
Since $X_q$ have disjoint relative interiors and $\Haus^1$ does not charge their endpoints, and since $\mu_i$ is a Borel probability measure on our Polish space, it follows by \cite[Theorem A.7]{BianchiniCaravenna} (cf. \cite[Theorem 6.18]{CavallettiEMilman-LocalToGlobal}) that the disintegration must be essentially unique, meaning that for any other disintegration:
\[
\mu_i = \int_{\Q} \tilde \mm_q \qq(dq) ,
\]
with $\tilde \mm_q$ concentrated on $\relint X_q$ for $\qq$-a.e. $q$, we must have $\tilde \mm_q =  \frac{d\mu_i}{d\mm} h_q \Haus^1\llcorner_{X_q}$ for $\qq$-a.e. $q$. 
Comparing (\ref{eq:mui2}) with the definition of $\mu_i$ from (\ref{eq:mui}), it immediately follows that $\frac{d\mu_i}{d\mm} = \rho_i$ $\Haus^1\llcorner_{X_q}$-a.e. for $\qq$-a.e. $q$, which by the disintegration (\ref{eq:disintegration}) means that $\frac{d\mu_i}{d\mm} = \rho_i$ $\mm$-a.e., and in particular estabishes the $\mm$-measurability of $\rho_i$. 
\end{remark}

\subsection{Characterization of one-dimensional case}

Before concluding this section, it is worth noting that, at least in the one-dimensional setting, Theorem \ref{thm:gen-loc} admits the following (standard) converse.

\begin{lemma} \label{lem:1D-densities}
Let $N$, $\sigma_0,\sigma_1$ be as in Theorem \ref{thm:gen-loc}, and let $h : \R \rightarrow \R_+$ be continuous on its support.
Then the one-dimensional metric-measure space $(\R, |\cdot|,  \mm = h \mathcal L^1)$ satisfies (\ref{eq:gen-loc-assumption}) if and only if $h$ satisfies:
\begin{equation} \label{eq:1D-density}
h^{\frac{1}{N-1}}((1-t)x_0 + t x_1) \geq \sigma^{(1-t)}_0(|x_1-x_0|) h^{\frac{1}{N-1}}(x_0) + \sigma^{(t)}_1(|x_1-x_0|) h^{\frac{1}{N-1}}(x_1) ,
\end{equation}
for all $x_0, x_1 \in \supp h$ and $t\in (0,1)$.
\end{lemma}

\begin{proof}
The ``only if" direction follows immediately from the proof of Theorem \ref{thm:gen-loc} (after localization to dimension one, the $\MCP(K',N')$ assumption was only used there to guarantee that the density $h$ is continuous on its support). The ``if" direction is standard, but for completeness, we sketch the proof. Let $\rho_0,\rho_1 : \supp h \rightarrow \R_+$ be two probability densities w.r.t. $\mm$ so that $\mu_0 := \rho_0 \mm$ and $\mu_1 := \rho_1 \mm$ are in $\P_c(\R)$. The $W_2$ optimal transport between $\mu_0$ and $\mu_1$ is obtained by a monotone map $T_1 : \supp h \rightarrow \supp h$, and by the change-of-variables formula, we have $J_1(x_0) := T'_1(x_0) = \frac{\rho_0(x_0) h(x_0)}{\rho_1(x_1) h(x_1)}$ for $\mu_0$-a.e. $x_0$, where we denote $x_1 := T_1(x_0)$. The $W_2$ geodesic $\mu_t := \rho_t \mm$ is obtained by pushing forward $\mu_0$ via $T_t(x) = (1-t) x + t T_1(x)$, and so by the change-of-variables formula, we have for each $t \in [0,1]$ that for $\mu_0$-a.e. $x_0$:
\[
J_t(x_0) := (1-t) + t J_1(x_0) = \frac{\rho_0(x_0) h(x_0)}{\rho_t(x_t) h(x_t)} ,
\]
where $x_t := T_t(x_0) = (1-t) x_0 + t x_1$. Abbreviating $C^{-1}_{x_0} := \rho_0(x_0) h(x_0)$, it follows that for $\mu_0$-a.e. $x_0$, by (\ref{eq:1D-density}) and H\"{o}lder's inequality:
\begin{align*}
& (C_{x_0} \rho_t(x_t))^{-\frac{1}{N}} = J_t^{\frac{1}{N}}(x_0) h^{\frac{1}{N}}(x_t) \\
& \geq \brac{(1-t) J_0(x_0) + t J_1(x_0)}^{\frac{1}{N}}  \brac{\sigma^{(1-t)}_0 (|x_1-x_0|) h^{\frac{1}{N-1}}(x_0) + \sigma^{(t)}_1 (|x_1-x_0|) h^{\frac{1}{N-1}}(x_1)}^{\frac{N-1}{N}} \\
& \geq  \brac{(1-t) J_0(x_0)}^{\frac{1}{N}} \brac{\sigma^{(1-t)}_{0} (|x_1-x_0|) h^{\frac{1}{N-1}}(x_0)}^{\frac{N-1}{N}} + (t J_1(x_0))^{\frac{1}{N}} \brac{ \sigma^{(t)}_{1} (|x_1-x_0|) h^{\frac{1}{N-1}}(x_1) }^{\frac{N-1}{N}}  \\
& = (1-t)^{\frac{1}{N}} \sigma^{(1-t)}_{0} (|x_1-x_0|)^{\frac{N-1}{N}} (C_{x_0} \rho_0(x_0))^{-\frac{1}{N}} + t^{\frac{1}{N}} \sigma^{(t)}_{1} (|x_1-x_0|) (C_{x_0} \rho_1(x_1))^{-\frac{1}{N}} ,
\end{align*}
establishing (\ref{eq:gen-loc-assumption}).
\end{proof}

\begin{remark} \label{rem:continuity-not-needed}
By employing Lebesgue's differentiation theorem and allowing to modify $h$ on a null-set, one may show (e.g. as in \cite[Lemma 3.3.10]{CalderonThesis}) that Lemma \ref{lem:1D-densities} remains valid for general $h \in L^1_{loc}(\R)$, without requiring continuity. We refrain from this generality here, as it will not be needed. 
\end{remark}

\section{One-dimensional $\QCD$ densities} \label{sec:1D-QCD}

\begin{definition}[One-dimensional $\QCD$ density]
Let $K \in \R$, $N \in (1,\infty)$ and $Q \geq 1$. 
We say that a function $h : \R \rightarrow \R_+$ which is continuous on its support is a $\QCD(Q,K,N)$ density if:
\begin{equation}\label{eq:qcd}
h^{\frac{1}{N-1}}(tx_1+(1-t)x_0) \geq \frac{1}{Q^{\frac{1}{N-1}}} \brac{\sigma^{(1-t)}_{K, N-1} (|x_1-x_0|) h^{\frac{1}{N-1}}(x_0) + \sigma^{(t)}_{K, N-1} (|x_1-x_0|) h^{\frac{1}{N-1}}(x_1)} ,
\end{equation}
for all $x_0, x_1 \in \supp h$ and $t\in (0,1)$. 
\end{definition}

\begin{remark}
Clearly, the support of a $\QCD$ density $h$ is always an interval and $h$ is strictly positive in its interior. Note that a function $h$ satisfying (\ref{eq:qcd}) with $Q > 1$ may in general be discontinuous at every point of its support, and hence we in addition require continuity above. 
\end{remark}
\begin{remark}
When $Q=1$, $h$ as above is said to be a $\CD(K,N)$ density. In this case, there is no need to \emph{a-priori} assume that $h$ is continuous on its support; any $h : \R \rightarrow \R_+$ satisfying (\ref{eq:qcd}) with $Q=1$ is automatically lower semi-continuous on its support and continuous in its interior (see e.g. \cite[Appendix A]{CavallettiEMilman-LocalToGlobal}), and so up to modifying the value of $h$ at the end-points, such an $h$ is already continuous. 
\end{remark}

Applying Lemma \ref{lem:1D-densities} with $\sigma_i^{(t)}(\theta) = \frac{1}{Q^{\frac{1}{N-1}}} \sigma^{(t)}_{K,N-1}$, we immediately obtain:
\begin{corollary} \label{cor:QCD-density}
Given $K \in \R$, $N \in (1,\infty)$, $Q \geq 1$ and a function $h : \R \rightarrow \R_+$ which is continuous on its support,
the one-dimensional metric-measure space $(\R, |\cdot|,  h \mathcal L^1)$ satisfies $\QCD(Q,K,N)$ if and only if $h$ is a $\QCD(Q,K,N)$ density. 
\end{corollary}

Note that when $K > 0$, $\R_+ \ni \theta \mapsto \sigma^{(t)}_{K,N-1}(\theta)$ is not continuous for $t=0,1$, as it jumps from $0,1$ (respectively) to $+\infty$ at $\theta = D_{K,N}$. However, the values $t=0,1$ were (deliberately) excluded from consideration in all of the statements of the previous section, and so Lemma \ref{lem:1D-densities} applies. 

\medskip
For later use, we introduce the following one-dimensional members of the family $\QCD_{reg}(Q,K,N)$ defined in Section \ref{sec:results}:
\begin{definition}[$\QCD_1(Q,K,N)$]
We denote by $\QCD_1(Q,K,N)$ the one-dimensional metric-measure spaces $(\R,\abs{\cdot},h \L^1)$ satisfying $\QCD(Q,K,N)$ and $\MCP(K',N')$ for some $K' \in \R$ and $N' \in (1,\infty)$. 
\end{definition}

As usual, note that when $Q=1$, $\QCD_1(1,K,N)$ coincides with $\CD_1(K,N)$, defined in Section \ref{sec:results}. 
We can now remove the continuity assumption in Corollary \ref{cor:QCD-density} (without invoking Remark \ref{rem:continuity-not-needed}), and obtain it as part of the conclusion: 
\begin{corollary} \label{cor:QCD-density-2}
Given $K \in \R$, $N \in (1,\infty)$, $Q \geq 1$ and $h \in L^1_{loc}(\R)$,  $(\R, |\cdot|,  h \mathcal L^1) \in \QCD_1(Q,K,N)$ if and only if (up to modification on a null-set) $h$ is both a $\QCD(Q,K,N)$ and  $\MCP(K',N')$ density, for some $K' \in \R$ and $N' \in (1,\infty)$. 
\end{corollary}
\begin{proof}
The ``if" direction follows from the ``if" directions of Corollary \ref{cor:QCD-density} and Lemma \ref{lem:MCP-density-1}. The ``only if" direction follows by first using the $\MCP(K',N')$ property of the space to invoke Lemmas \ref{lem:MCP-density-1} and \ref{lem:MCP-density-2} and conclude that up to modification on a null-set, $h$ is continuous on its support, and then applying the ``only if" direction of Corollary \ref{cor:QCD-density}. 
\end{proof}

\subsection{One-dimensional $\QCD$ and $\CD$ densities are equivalent}

By Theorem \ref{thm:gen-loc} and Corollary \ref{cor:QCD-density}, we can already reduce the study of any property of $\QCD$ spaces which is amenable to localization to the one-dimensional case. To treat the one-dimensional case, our second main observation in this work is as follows:

\begin{proposition}[One-dimensional $\QCD$ and $\CD$ densities are equivalent] \label{prop:QCD-CD}
$h$ is a $\QCD(Q,K,N)$ density iff there exists a $\CD(K,N)$ density $f$ so that:
\[
h \leq f \leq Q h .
\]
\end{proposition}

Contrary to the results of the previous section, Proposition \ref{prop:QCD-CD} is rather particular to the functions $\sigma_i^{(t)} = \sigma^{(t)}_{K,N-1}$. The reason is that $\sigma(t) = \sigma^{(t)}_{K,N-1}(\theta)$ (for $\theta < D_{K,N}$) satisfies the following second-order ODE:
\begin{equation} \label{eq:sigma-ODE}
\sigma''(t) + \theta^2 \frac{K}{N-1} \sigma(t) = 0 \;\; \text{on $t \in [0,1]$} \; , \; \sigma(0) = 0 \; , \; \sigma(1) = 1 . 
\end{equation}
Consequently, we will construct $f$ above as a ``$\CD(K,N)$ upper envelope" of $h$. For the proof, we will require the following:

\begin{definition}[$\CD(K,N)$ model density]
A function $f_m : \R \rightarrow \R_+$ which is smooth on its support and satisfies:
\begin{equation} \label{eq:CDKN-model}
(f_m^{\frac{1}{N-1}})'' (t) + \frac{K}{N-1} f_m^{\frac{1}{N-1}}(t) = 0 \text{ on $\supp f_m$ } 
\end{equation}
is called a $\CD(K,N)$ model density. 
\end{definition}
Using (\ref{eq:sigma-ODE}), one immediately verifies that a $\CD(K,N)$ model density is a $\CD(K,N)$ density which satisfies (\ref{eq:qcd}) with equality (and $Q=1$). Note that the maximal interval $I_{f_m}$ on which a solution to (\ref{eq:CDKN-model}) exists and coincides with $f_m$ on $\supp f_m$ is of diameter $D_{K,N}$, and hence $\diam(\supp f_m) \leq D_{K,N}$. We will say that $f_m$ is of maximal support if $\supp f_m = I_{f_m}$; note that in that case, $f_m$ is continuous on the entire $\R$. For more on the well-known differential characterization of $\CD(K,N)$ densities as satisfying (\ref{eq:CDKN-model}) with $\leq 0$ instead of $=0$ (in the sense of distributions) we refer to \cite[Appendix A]{CavallettiEMilman-LocalToGlobal}.

Note that contrary to $\CD(K,N)$ densities, $\QCD(Q,K,N)$ densities do not and cannot satisfy any differential characterization whenever $Q > 1$. To see this, take any $\CD(K,N)$ density $f$ supported on an interval $I$ of positive length, and multiply it by any continuous function $p$ which oscillates on $I$ between the values of $1$ and $Q$; the resulting density $h = f p$ is a $\QCD(Q,K,N)$ density by (the trivial direction of) Proposition \ref{prop:QCD-CD}. Obviously, by making $p$ oscillate as violently as one desires, no differential characterization of $\QCD(Q,K,N)$ densities is possible, and furthermore, it is possible to arrange so that $h$ does not satisfy any $\CD(K',N')$ condition for any $K' \in \R$ and $N' \in (1,\infty)$. A concrete example of a density which satisfies $\QCD(2,0,2)$ but not $\CD(K',N')$ for any $K' \in \R$ and $N' \in (1,\infty)$ is given e.g. by $h(x) = 1 + |x|$ on the interval $[-1,1]$; the latter is easily seen after noting that the distributional second derivative of $h$ on $[-1,1]$ is the delta-measure $2 \delta_{0}$.

\begin{proof}[Proof of Proposition \ref{prop:QCD-CD}]
The ``if" direction is trivial by using that $f$ is a $\CD(K,N)$ density and passing from $f$ to $h$ using $h \leq f \leq Q h$. For the ``only if" direction, 
let $h$ be a $\QCD(Q,K,N)$ density. Its support is a closed interval, and we may assume it is non-empty (and thus of positive length), otherwise there is nothing to prove. Define:
\[
\bar f := \inf \{ f_m \; ; \; f_m \text{ is a $\CD(K,N)$ model density with $\supp f_m = \supp h$ and $f_m \geq h$ } \} ,
\]
where the infimum is interpreted pointwise. Note that by definition of $\CD(K,N)$ density, the pointwise infimum of a set of $\CD(K,N)$ densities having common support $I \subset \R$ is itself a $\CD(K,N)$ density (whose support is in general a subset of $I$); note that the infimum will automatically be continuous on $I$ since it is upper semi-continuous (being an infimum of continuous functions) and lower semi-continuous (satisfying (\ref{eq:qcd}) with $Q=1$). Hence, assuming the infimum above is over a non-empty set, then $\bar f$ is a $\CD(K,N)$ density satisfying $\bar f \geq h$, and in particular $\supp \bar f = \supp h$.

In addition, define:
\[
\ubar f(x) := \sup \left \{  \begin{array}{l} \brac{\sigma^{(1-t)}_{K, N-1} (|x_1-x_0|) h^{\frac{1}{N-1}}(x_0) + \sigma^{(t)}_{K, N-1} (|x_1-x_0|) h^{\frac{1}{N-1}}(x_1)}^{N-1} \;  ;   \\ 
  (1-t) x_0 + t x_1 = x ~,~ t \in [0,1] ~,~ x_0,x_1 \in \supp h  \end{array} \right \} ,
\]
if $x \in \supp h$ and $\ubar f(x) = 0$ otherwise. Note that by definition of $\QCD$ density, $\ubar f \leq Q h$.

We will show that $\bar f = \ubar f$ on $\intt \supp h$, and so setting $f = \bar f$, will conclude that $f$ is a $\CD(K,N)$ density on $\supp h$ with $h \leq f \leq Qh$ on $\intt \supp h$ (and hence on $\supp h$ by continuity of $h$), as desired. To this end, we require the following:

\begin{lemma}
For all $x \in \intt \supp h$, there exists a $\CD(K,N)$ model density $f_m$ so that $f_m(x) = \ubar f(x)$ and $f_m \geq h$. 
\end{lemma}

Once this lemma is established, it first follows that the infimum in the definition of $\bar f$ is indeed over a non-empty set (by choosing any $x \in \intt \supp h$, applying the lemma and restricting $f_m$ to $\supp h$). Moreover, the lemma immediately implies that $\bar f \leq \ubar f$ on $\intt \supp h$. On the other hand, we also have $\bar f \geq \ubar f$ on $\supp h$, since if $f_m$ is a $\CD(K,N)$ model density with $f_m \geq h$, then for any $t \in [0,1]$ and  $x ,x_0 , x_1 \in \supp h$ so that $x = (1-t) x_0 + t x_1$, we have:
\begin{align*}
f_m^{\frac{1}{N-1}}(x) & = \sigma^{(1-t)}_{K, N-1} (|x_1-x_0|) f_m^{\frac{1}{N-1}}(x_0) + \sigma^{(t)}_{K, N-1} (|x_1-x_0|) f_m^{\frac{1}{N-1}}(x_1) \\
& \geq \sigma^{(1-t)}_{K, N-1} (|x_1-x_0|) h^{\frac{1}{N-1}}(x_0) + \sigma^{(t)}_{K, N-1} (|x_1-x_0|) h^{\frac{1}{N-1}}(x_1) ,
\end{align*}
and so taking supremum over $t,x_0,x_1$ as above, it follows that $f_m(x) \geq \ubar f(x)$, and taking infimum over $f_m$ as above, we indeed verify that $\bar f \geq  \ubar f$. This implies that $\bar f = \ubar f$ on $\intt \supp h$, and so all that remains is to establish the lemma. 

Given $x \in \intt \supp h$, assume in the contrapositive that there is no $\CD(K,N)$ model density $f_m$ so that $f_m(x) = \ubar f(x)$ and $f_m \geq h$. Hence, for any $\CD(K,N)$ model density $f_m$ of maximal support (and therefore continuous on $\supp h$) so that $f_m(x) = \ubar f(x)$, either there exists $x_1 > x$ so that $0 < f_m(x_1) < h(x_1)$ or there exists $x_0 < x$ so that $0< f_m(x_0) < h(x_0)$, but it is impossible that both possibilities occur simultaneously, since otherwise, as $x_0,x_1 \in \supp f_m \cap \supp h$, we would have (for $t \in (0,1)$ so that $x = (1-t) x_0 + t x_1$): 
\begin{align*}
f_m^{\frac{1}{N-1}}(x) & = \sigma^{(1-t)}_{K, N-1} (|x_1-x_0|) f_m^{\frac{1}{N-1}}(x_0) + \sigma^{(t)}_{K, N-1} (|x_1-x_0|) f_m^{\frac{1}{N-1}}(x_1) \\
& < \sigma^{(1-t)}_{K, N-1} (|x_1-x_0|) h^{\frac{1}{N-1}}(x_0) + \sigma^{(t)}_{K, N-1} (|x_1-x_0|) h^{\frac{1}{N-1}}(x_1) \leq \ubar f^{\frac{1}{N-1}}(x) ,
\end{align*}
a contradiction. Let us denote the first possibility above by $R$ and the second by $L$.

By the second order ODE description (\ref{eq:CDKN-model}), the set of $\CD(K,N)$ model densities $f_m$ of maximal support with a given value of $f_m(x)$ is parametrized by its slope $s = f_m'(x) \in \R$, and varies continuously in $s$. Consequently, $L$ and $R$ are complementing open conditions with respect to $s \in \R$, and so by connectedness of $\R$, either $L$ or $R$ must hold for \emph{all} $f_m$ with $f_m(x) = \ubar f(x)$ simultaneously. But this is impossible:
fixing $x_0 < x < x_1$ so that $x_0,x_1 \in \intt \supp h$ (i.e. $h(x_0),h(x_1) > 0$), it is immediate to show (see \cite[Lemma 3.1]{EMilmanNegativeDimension}) that $f_m(x_0) \rightarrow 0$ when $s \rightarrow +\infty$ and that $f_m(x_1) \rightarrow 0$ when $s \rightarrow -\infty$, and so both possibilities $L$ and $R$ can occur, a contradiction. Note that this argument is also valid when $K > 0$, even though the (maximal) support of $f_m$ may not contain $\supp h$. 

This concludes the proof of the lemma, and hence of the proposition. 
\end{proof}

\section{Functional Inequalities on $\QCD$ spaces} \label{sec:inqs}

\subsection{Equivalent Formulation, Monotonicity and Stability}

We begin this section by rewriting the $L^p$-Poincar\'e and log-Sobolev inequalities we consider in this work in an equivalent form. 
Note that since $\Omega$ is always assumed bounded, $(\supp(\mm),\d)$ is proper by the underlying $\MCP(K',N')$ assumption, $\mm$ is locally finite, and the test function $f$ is locally Lipschitz, then all integrals involved in these inequalities are necessarily finite. We formulate the inequalities a bit more generally, using a bounded $\Lambda \supset \Omega$ instead of $\conv(\Omega)$ on the energy side of the inequalities. 

\begin{itemize}
\item The $L^p$-Poincar\'e constant $\lambda_p[(X,\d,\mm),\Omega,\Lambda]$ is defined as the best constant $\lambda_p$ so that
for any (locally) Lipschitz function $f : (X,\d) \rightarrow \R$:
\begin{equation} \label{eq:def-Lp}
\int_{\Omega} |f|^{p-2} f \mm = 0  \;\; \Rightarrow \;\; \lambda_{p} \int_{\Omega} |f|^p \mm \leq \int_{\Lambda} |\nabla_{X} f|^p \mm .
\end{equation}
Note that it coincides with the best constant $\lambda_p$ so that for any (locally) Lipschitz function $f : (X,\d) \rightarrow \R$:
\[
\lambda_p \min_{c \in \R} \int_{\Omega} |f - c|^p \mm \leq \int_{\Lambda} |\nabla_X f|^p \mm . 
\]
Indeed, this is immediate after noting that the unique minimizing $c$ above (since $p \in (1,\infty)$) satisfies $\int_{\Omega} |f-c|^{p-2} (f-c) \mm = 0$, and of course $|\nabla_X f| = |\nabla_X (f-c)|$. 
\item The log-Sobolev constant $\lambda_{LS}[(X,\d,\mm),\Omega,\Lambda]$ is defined as the best constant $\lambda_{LS}$ so that 
for any (locally) Lipschitz function $f : (X,\d) \rightarrow \R$:
\begin{equation} \label{eq:def-LS}
\int_{\Omega} (f^2 - 1) \mm = 0  \;\; \Rightarrow \;\; \frac{\lambda_{LS}}{2} \int_{\Omega} f^2 \log(f^2) \mm \leq \int_{\Lambda} |\nabla_{X} f|^2 \mm .
\end{equation}
It coincides (when $\mm(\Omega) > 0$) with the best constant $\lambda_{LS}$ so that for any (locally) Lipschitz function $f : (X,\d) \rightarrow \R$:
\[
\lambda_{LS} \int_{\Omega} \brac{ \Phi(f^2)  - \Phi\brac{\frac{1}{\mm(\Omega)} \int_{\Omega} f^2 \mm } } \mm \leq \int_{\Lambda} |\nabla_X f|^2 \mm ,
\]
where $\Phi(x) := x \log(x)$. Indeed, this is immediate to check by applying (\ref{eq:def-LS}) to $f/\sqrt{c}$ with $c = \int_{\Omega} f^2 \mm / \mm(\Omega)$ whenever $c > 0$ on one hand, and noting that $\Phi(1) = 0$ on the other. Furthermore, the convexity of $\Phi : \R_+ \rightarrow \R$ ensures (see Holley--Stroock \cite{HolleyStroockPerturbationLemma} or the proof of \cite[Proposition 5.5]{Ledoux-Book}) that for all non-negative $g$ for which the integrals below are finite:
\[
\int_{\Omega} \brac{ \Phi(g)  - \Phi\brac{\frac{1}{\mm(\Omega)} \int_{\Omega} g \mm } } \mm = \inf_{t \in \R_+} \int_{\Omega} \brac{\Phi(g) - \Phi(t) - \Phi'(t) (g-t) } \mm ,
\]
and that the integrand on the right-hand-side is non-negative for each $t$. 
\end{itemize}

We conclude that we can express each of our functional inequalities (\ref{eq:def-Lp}) and (\ref{eq:def-LS}) in the form:
\begin{equation} \label{eq:gen-func}
\lambda_*[(X,\d,\mm),\Omega,\Lambda] \inf_{\alpha \in A} \int_{\Omega} F_\alpha(f) \mm \leq \int_{\Lambda} G(|\nabla_X f|) \mm \;\;\;  \forall \text{ locally Lipschitz $f$} ,
\end{equation}
for an appropriate $G$ and family $\{F_\alpha\}_{\alpha \in A}$ of \emph{non-negative functionals} (depending on $\lambda_* \in \{\lambda_p , \lambda_{LS} \}$), with identical best constants in either formulation. Two immediate crucial consequences are:

\begin{lemma} \label{lem:mon-stable}
The best constant $\lambda_*[(X,\d,\mm),\Omega,\Lambda]$ in (\ref{eq:gen-func}) satisfies:
\begin{enumerate}
\item Monotonicity: if $\Omega_2 \subset \Omega_1, \Lambda_2 \supset \Lambda_1$ then $\lambda_*[(X,\d,\mm),\Omega_2,\Lambda_2] \geq \lambda_*[(X,\d,\mm),\Omega_1,\Lambda_1]$. 
\item Stability: if $\mm_2 \leq c_1 \mm_1$ on $\Omega$ and $\mm_1 \leq c_2 \mm_2$ on $\Lambda$ then $\lambda_*[(X,\d,\mm_2),\Omega,\Lambda] \geq \frac{1}{c_1 c_2} \lambda_*[(X,\d,\mm_1),\Omega,\Lambda]$. 
\end{enumerate}
\end{lemma}

\subsection{One-dimensional case}

As an immediate corollary, we obtain:
\begin{corollary} \label{cor:lambda-1D}
For any family $\mathscr{X}$ of \emph{one-dimensional} metric measure spaces $(\R,|\cdot|,\mm)$ for which $\supp(\mm)$ is an interval and which is closed under restrictions to intervals, and for any $D \in (0,\infty)$, we have $\lambda_*[\mathscr{X},D] = \bar \lambda_*[\mathscr{X},D]$. In particular, this applies to $\mathscr{X} = \CD_1(K,N)$ and $\mathscr{X} = \QCD_1(Q,K,N)$. 
\end{corollary}
\begin{proof}
The inequality $\lambda_*[\mathscr{X},D] \leq \bar \lambda_*[\mathscr{X},D]$ always holds, so we just need to show the converse. Given $(\R,\abs{\cdot},\mm) \in \mathscr{X}$ and a closed $\Omega \subset \supp(\mm)$ of diameter at most $D$, the monotonicity assertion of Lemma \ref{lem:mon-stable} implies:
\begin{align*}
& \lambda_*[(\R,|\cdot|,\mm),\Omega,\conv(\Omega)] \geq \lambda_*[(\R,|\cdot|,\mm),\conv(\Omega),\conv(\Omega)] \\
& = \lambda_*[(\R,|\cdot|,\mm \llcorner_{\conv(\Omega)}),\conv(\Omega),\conv(\Omega)] \geq \bar \lambda_*[\mathscr{X},D] ,
\end{align*}
since $(\R,|\cdot|,\mm \llcorner_{\conv(\Omega)}) \in \mathscr{X}$ and $\supp(\mm \llcorner_{\conv(\Omega)}) = \conv(\Omega)$ is an interval of diameter at most $D$. Taking infimum over all $(\R,\abs{\cdot},\mm)$ and $\Omega$ as above concludes the proof. 
\end{proof}

Since $\conv(\Omega)$ is not necessarily geodesically convex in dimension greater than $1$,
we do not know how to extend the identification between $\lambda_*$ and $\bar \lambda_*$ asserted in Corollary \ref{cor:lambda-1D} to general families of metric-measure spaces. 
However, for families which admit localization to one-dimensional geodesics like $\CD_{reg}(K,N)$ or more generally $\QCD_{reg}(Q,K,N)$, we can in fact extend it as described in Theorem \ref{thm:QCD-lambda-loc} below. 

Together with Proposition \ref{prop:QCD-CD}, we can already conclude the one-dimensional case of Theorem \ref{thm:intro-main}:
\begin{theorem} \label{thm:main-1D}
For all $K \in \R$, $N \in (1,\infty)$, $D \in (0,\infty)$, $Q \geq 1$ and $\lambda_* \in \{ \lambda_p , \lambda_{LS} \}$:
\[
\bar \lambda_*[\CD_{1}(K,N),D] \geq \lambda_*[\QCD_{1}(Q,K,N),D] \geq \frac{1}{Q} \bar \lambda_*[\CD_{1}(K,N),D]  .
\]
\end{theorem}
\begin{proof}
The first inequality is trivial since $\CD_{1}(K,N) \subset \QCD_{1}(Q,K,N)$. Taking into account Corollary \ref{cor:lambda-1D}, it remains to establish:
\[
\bar \lambda_*[\QCD_{1}(Q,K,N),D] \geq \frac{1}{Q} \bar \lambda_*[\CD_{1}(K,N),D] . 
\]
Let $(\R,\abs{\cdot},h \L^1) \in \QCD_{1}(Q,K,N)$ with $I = \supp h$ having diameter at most $D$. By Corollary \ref{cor:QCD-density-2}, up to modifications on a null-set, $h$ is a $\QCD(Q,K,N)$ density. By Proposition \ref{prop:QCD-CD}, there exists a $\CD(K,N)$ density $f$ so that $h \leq f \leq Qh$. Consequently, the stability assertion of Lemma \ref{lem:mon-stable} implies that:
\[
\lambda_*[(\R,\abs{\cdot},h \L^1),I,I] \geq \frac{1}{Q} \lambda_*[(\R,\abs{\cdot},f \L^1),I,I] \geq \frac{1}{Q}  \bar \lambda_*[\CD_{1}(K,N),D]  . 
\]
Taking infimum over all $(\R,\abs{\cdot},h \L^1)$ as above concludes the proof. 
\end{proof}

\subsection{Localization}

It remains to establish:

\begin{theorem} \label{thm:QCD-lambda-loc}
For all $K \in \R$, $N \in (1,\infty)$, $D \in (0,\infty)$, $Q \geq 1$ and $\lambda_* \in \{ \lambda_p , \lambda_{LS} \}$:
\[
\bar \lambda_*[\QCD_{reg}(Q,K,N),D] = \lambda_*[\QCD_{reg}(Q,K,N),D] = \bar \lambda_*[\QCD_{1}(Q,K,N),D] = \lambda_*[\QCD_{1}(Q,K,N),D] .
\]
\end{theorem}

\noindent
In conjunction with Theorem \ref{thm:main-1D}, this will establish our main Theorem \ref{thm:intro-main}. 

\begin{proof}[Proof of Theorem \ref{thm:QCD-lambda-loc}]
Since $\QCD_1(Q,K,N) \subset \QCD_{reg}(Q,K,N)$ and $\bar \lambda_* \geq \lambda_*$ always, we trivially have:
\[
\bar \lambda_*[\QCD_{1}(Q,K,N),D] \geq \bar \lambda_*[\QCD_{reg}(Q,K,N),D] \geq \lambda_*[\QCD_{reg}(Q,K,N),D]  ,
\]
so it remains to establish that $\lambda_*[\QCD_{reg}(Q,K,N),D] \geq \bar \lambda_*[\QCD_{1}(Q,K,N),D]$ to close the chain of inequalities and conclude that they are in fact all equalities. 
Denote by $Z_* : \R \rightarrow \R$, $* \in \{ p , LS \}$, the function $Z_p(t) := |t|^{p-2} t$ and $Z_{LS}(t) := t^2 - 1$. 
Given $(X,\d,\mm) \in \QCD_{reg}(Q,K,N)$, a closed $\Omega \subset \supp(\mm)$ with $\diam(\Omega) \leq D$, and a (locally) Lipschitz function $f$ on $(X,\d)$ with $\int_\Omega Z_*(f) \mm = 0$, set $g = Z_*(f) 1_{\Omega}$. As $\Omega$ is bounded, $(\supp(\mm),\d)$ is proper by $\MCP(K',N')$, and $\mm$ is locally finite, the integrability assumption $\int_X |g(x)| d(x,x_0) \mm(dx) < \infty$ is clearly satisfied, and we may apply the Generalized Localization Theorem \ref{thm:gen-loc} with the $\QCD$ interpolation weights $\sigma_i^{(t)}(\theta) = Q^{-\frac{1}{N-1}}\sigma^{(t)}_{K,N-1}$ (recalling in addition Corollary \ref{cor:QCD-density}). 

It follows that there exists an $\mm$-measurable subset $\T \subset X$ and a family $\{X_{q}\}_{q \in \Q} \subset X$ so that the following disintegration of $\mm\llcorner_{\T}$ on $\{X_{q}\}_{q \in \Q}$ holds:
\[
\mm\llcorner_{\T} = \int_{\Q} \mm_{q} \, \qq(dq)  ~,~ \qq(\Q) = 1 ,
\]  
and for $\qq$-a.e. $q \in \Q$:
\begin{enumerate}
\item $X_q$ is a closed geodesic in $(X,\d)$. 
\item $\mm_q$ is a Radon measure supported on $X_q$ with $\mm_q \ll  \Haus^1 \llcorner_{X_q}$.
\item $\int_{X_q \cap \Omega} Z_*(f) \mm_q = \int  g \mm_q = 0$.
\item $(X_{q}, \d,\mm_{q})$ verifies $\MCP(K',N')$.
\item $(X_{q}, \d,\mm_{q})$ verifies $\QCD(Q,K,N)$.
\end{enumerate}
In addition, $g \equiv 0$ $\mm$-a.e. on $X \setminus \T$, implying that $Z_*(f) \equiv 0$  $\mm$-a.e. on $\Omega \setminus  \T$.

Since $\supp(g \mm) \subset \Omega$, we know that $\diam(\supp(g \mm)) \leq D$. 
Let $q \in \Q$ be such that all of the above properties hold, and denote:
\[
L_q := \conv_{X_q}(X_q \cap \supp(g \mm)) \; ,
\]
where the geodesic (convex) hull is taken in the metric space $(X_{q}, \d)$ which is isometric to a closed subinterval of $(\R,|\cdot|)$. It follows that $\diam(L_q) \leq D$, and we have:
\begin{equation} \label{eq:Lq}
X_q \cap \supp(g \mm) \subset L_q \subset X_q \cap \conv(\supp(g \mm))   .
\end{equation}
Since $\mm\llcorner_{\T}(\{ g \neq 0\} \setminus \supp(g \mm)) = 0$, the above disintegration and Fubini's theorem imply that for $\qq$-a.e. $q \in \Q$, $g \equiv 0$ $\mm_q$-a.e. on $X \setminus \supp(g \mm)$ and in particular on $X_q \setminus L_q$. It follows by property (3) that for $\qq$-a.e. $q \in \Q$:
\begin{enumerate}
\setcounter{enumi}{5}
\item $Z_*(f) \equiv 0$ $\mm_q$-a.e. on $X_q \cap \Omega \setminus (L_q \cap \Omega)$ and $\int_{L_q\cap \Omega} Z_*(f) \mm_q = 0$. 
\end{enumerate}
We therefore add this requirement from $q$ to our previous requirements, as they all hold for $\qq$-a.e. $q \in \Q$.

Since the $\QCD(Q,K,N)$ and $\MCP(K',N')$ conditions are closed under restrictions onto geodesically convex subsets, it follows that $(L_q , \d, \mm_q\llcorner_{L_q})$ verifies both conditions; however, since $\Omega$ was not assumed to be geodesically convex, note that $(L_q \cap \Omega, \d, \mm_q\llcorner_{(L_q \cap \Omega)})$ \textbf{may not} satisfy $\QCD(Q,K,N)$ nor $\MCP(K',N')$.
Nevertheless, by the monotonicity property established in Lemma \ref{lem:mon-stable}:
\[ 
\lambda_*[(L_q,\d, \mm_q \llcorner_{L_q}), L_q \cap \Omega, L_q]  \geq \lambda_*[(L_q,\d,\mm_q \llcorner_{L_q}), L_q , L_q] \geq \bar \lambda_*[\QCD_1(Q,K,N),D] ,
\]
where the last inequality is due to the fact that $(L_q , \d, \mm_q\llcorner_{L_q})$ is (isometric to) a one-dimensional metric-measure space satisfying $\QCD(Q,K,N)$ and $\MCP(K',N')$ and $\diam(L_q) \leq D$.

Since $\int_{L_q\cap \Omega} Z_*(f) \mm_q = 0$ by property (6), we may revert back from the infimum formulation (\ref{eq:gen-func}) of our functional inequality to the standard one in (\ref{eq:def-Lp}) or (\ref{eq:def-LS}) for $\lambda_* \in \{ \lambda_p , \lambda_{LS} \}$, respectively. We conclude that:
\[
\bar \lambda_p[\QCD_1(Q,K,N),D] \int_{L_q \cap \Omega} |f|^p \mm_q \leq \int_{L_q} |\nabla_{L_q} f|^p \mm_q ,
\]
in the first case, and:
\[
 \frac{\bar \lambda_{LS}[\QCD_1(Q,K,N),D]}{2}  \int_{L_q \cap \Omega} f^2 \log(f^2) \mm_q \leq \int_{L_q} |\nabla_{L_q} f|^2 \mm_q ,
\]
in the second. Recall that $Z_*(f) = 0$ $\mm_q$-a.e. on $X_q \cap \Omega \setminus (L_q \cap \Omega)$ by property (6), and hence the integrand on the left-hand-sides above vanishes $\mm_q$-a.e. on $X_q \cap \Omega \setminus (L_q \cap \Omega)$ (also in the $LS$ case, since $Z_{LS}(f) = 0$ iff $f^2 = 1$ iff $\log(f^2) = 0$). It follows that we may enlarge the domain of integration on the left-hand-sides to $X_q \cap \Omega$; on the right-hand-sides we may enlarge the domain of integration to $X_q \cap \conv(\supp(g \mm))$ thanks to the non-negativity of the integrand and (\ref{eq:Lq}). 

Using $|\nabla_{L_q} f|\leq |\nabla_X f|$ and integrating the resulting inequalities with respect to $\qq$, we deduce from the disintegration formula that:
\[
 \bar \lambda_p[\QCD_1(Q,K,N),D]  \int_{\T \cap \Omega} |f|^p \mm \leq  \int_{\T \cap \conv(\supp(g \mm))} |\nabla_{X} f|^p \mm ,
\]
and
\[
  \frac{\bar \lambda_{LS}[\QCD_1(Q,K,N),D]}{2} \int_{\T \cap \Omega} f^2 \log(f^2) \mm \leq  \int_{\T \cap \conv(\supp(g \mm))} |\nabla_{X} f|^2 \mm ,
\]
respectively. Recalling that $Z_*(f) \equiv 0$ $\mm$-a.e. on $\Omega \setminus \T$, we may enlarge as before the domain of integration on the left-hand-sides to $\Omega$; on the right-hand-sides we may enlarge it to $\conv(\Omega) \supset \conv(\supp (g \mm))$. This establishes that $\lambda_*[\QCD_{reg}(Q,K,N),D] \geq \bar \lambda_*[\QCD_{1}(Q,K,N),D]$, thereby concluding the proof. 
\end{proof}

\section{Concluding Remarks} \label{sec:conclude}

\subsection{Curvature Geodesic-Topological Dimension Condition}
Before concluding, we mention an alternative path for deriving the exact same results we obtain in this work, which is more tailored to the ideal sub-Riemannian setting. 

\begin{definition}[Curvature Geodesic-Topological Dimension condition $\CGTD(K,N,n)$]
A Monge space $(X,\d,\mm)$ is said to satisfy the $\CGTD(K,N,n)$ condition, $K \in \R$, $n \in [1,\infty)$, $n \leq N \in (1,\infty)$, if for all 
$\mu_0,\mu_1 \in P_c(X)$ with $\mu_0,\mu_1 \ll \mm$ and for all $t \in (0,1)$:
\begin{multline*}
\rho^{-\frac{1}{n}}_t(\gamma_t) \geq \tau^{(1-t)}_{K,N}(\d(\gamma_0,\gamma_1))^{\frac{N}{n}} \rho^{-\frac{1}{n}}_0(\gamma_0) + \tau^{(t)}_{K,N}(\d(\gamma_0,\gamma_1))^{\frac{N}{n}} \rho^{-\frac{1}{n}}_1(\gamma_1) \\
 \text{for $\nu$-a.e. $\gamma \in \geo(X,\d)$}. 
\end{multline*}
\end{definition}

Note that the $\CGTD(K,N,n)$ condition simultaneously implies both the $\MCP(K,N)$ condition (by dropping the right-most term above), and the $\QCD(2^{N-n},K,N)$ condition (by applying Jensen's inequality as in the proof of Proposition \ref{prop:MCP-QCD}). Repeating the argument in Section \ref{sec:results}, Theorem \ref{thm:interpolation} implies that the $\MCP(K,N)$ condition on an ideal $n$-dimensional sub-Riemannian manifold automatically self-improves to $\CGTD(K,N,n)$, and so all the ideal $n$-dimensional sub-Riemannian manifolds mentioned in Subsection \ref{subsec:QCD} satisfy $\CGTD(0,N,n)$ for some appropriate $N > n$. 

We may then apply the general localization Theorem \ref{thm:gen-loc} to deduce that the $\CGTD(K,N,n)$ condition localizes to one-dimensional geodesics, and so it is enough to study the properties of one-dimensional $\CGTD(K,N,n)$ densities $h$, which by Lemma \ref{lem:1D-densities} are characterized by:
\[
h^{\frac{1}{n-1}}(tx_1+(1-t)x_0) \geq \sigma^{(1-t)}_{K,N}(|x_1-x_0|)^{\frac{N-1}{n-1}} h^{\frac{1}{n-1}}(x_0) + \sigma^{(1-t)}_{K,N}(|x_1-x_0|)^{\frac{N-1}{n-1}} h^{\frac{1}{n-1}}(x_1) ,
\]
for all $x_0, x_1 \in \supp h$ and $t\in (0,1)$. Note that the case $n=1$ is understood in the limiting sense, namely as taking the maximum between the two terms on the right and thus recovering the $\MCP(K,N)$ density characterization. We see again, now on the level of one-dimensional densities, that a $\CGTD(K,N,n)$ density is simultaneously both an $\MCP(K,N)$ density (by dropping the right-most term above) and a $\QCD(2^{N-n},K,N)$ density (by Jensen's inequality). 

Repeating the argument of Section \ref{sec:inqs}, we immediately deduce that $\lambda_*(\CGTD(K,N,n),D) = \lambda_*(\CGTD_1(K,N,n),D) = \bar \lambda_*(\CGTD_1(K,N,n),D)$ (using the obvious analogues of our usual definitions and notation).

While this approach has the clear advantage of providing us with more information on the resulting one-dimensional densities after localization, we do not know how to use this additional information for the study of functional inequalities beyond what the $\QCD(2^{N-n},K,N)$ condition tells us, namely that there is an equivalent $\CD(K,N)$ density $f$ so that $h \leq f \leq 2^{N-n} h$, so that $\bar \lambda_*(\CGTD_1(K,N,n),D) \geq \frac{1}{2^{N-n}} \bar \lambda_*(\CD_1(K,N),D)$, thus arriving to the same conclusion as before. For this reason, we have chosen to present our results using the more general $\QCD$ condition, in the hope that it would also be applicable in more general settings beyond the sub-Riemannian one, when the $\CGTD$ condition is inapplicable. 

\subsection{Additional properties and variants}

Continuing in the same vein,  one can engage in a more comprehensive study of the $\QCD$ or $\CGTD$ conditions: determining what would be a good definition without a-priori assuming that the space is Monge or essentially non-branching, studying the stability of the resulting definition under measured Gromov--Hausdorff convergence and tensorization, rewriting it in terms of the $N$-Renyi entropy, extending the definition to include $N = \infty$, etc...~(in analogy to the Lott--Sturm--Villani program for the $\CD$ case). One can also introduce the $\RQCD$ and $\RCGTD$ conditions, in analogy to the $\RCD$ condition, by adding the assumption that the space is infinitesimally Hilbertian \cite{AGS-RCD, Gigli-MAMS-differential}, 
as it is known that sub-Riemannian Carnot groups are indeed infinitesimally Hilbertian \cite{LucicPasqualetto}.
Another interesting direction suggested by the referee is to investigate the relation between $\QCD(Q,K,N)$ or $\CGTD(K,N,n)$ and the Baudoin--Garofalo $\CD(\rho_1,\rho_2,\kappa,m)$ condition \cite{BaudoinGarofalo-SubRiemannianCD}. 
 We refrain from pursuing these directions here.

\subsection{Optimality of $Q = 2^{N-n}$ and Brunn--Minkowski Inequalities} \label{subsec:BM}

It was shown in the various references mentioned in Subsection \ref{subsec:QCD} that the corresponding sub-Riemannian manifolds satisfy $\MCP(0,N)$ with $N$ being best possible (i.e. minimal). It is also clear from the application of the (optimal) Jensen inequality (as in the proof of Proposition \ref{prop:MCP-QCD}) that the constant $Q = 2^{N-n}$ is best possible when transitioning from the $\CGTD(K,N,n)$ condition to the $\QCD(Q,K,N)$ one. However, one may wonder whether the overall optimality of the constant $2^{N-n}$ is lost when transitioning from the optimal $\MCP(0,N)$ condition to the $\QCD(2^{N-n} , 0 , N)$ one. We mention here that this is not the case and that the value $Q = 2^{N-n}$ is indeed \emph{optimal} (i.e. minimal) in the $\QCD(Q,0,N)$ condition, at least whenever the parameter $N$ from the $\MCP(0,N)$ condition coincides with the minimal geodesic dimension $\mathcal{N}$ (see \cite[Theorem 5]{BarilariRizzi-Interpolation} for the precise definition of the latter) -- by \cite{Rizzi-MCPonCarnot,BarilariRizzi-MCPonHTypeCarnot,BKS-CorankOneCarnot,LeeLiZelenko,BarilariRizzi-Interpolation,BarilariRizzi-BE}, this is the case for generalized H-type and corank $1$ Carnot groups, the Grushin plane, and Sasakian and $3$-Sasakian manifolds (under appropriate curvature lower bounds). 

To see the aforementioned optimality, observe that a standard application of the localization argument from the previous section would verify that the $\QCD(Q,0,N)$ condition implies the following ``quasi Brunn--Minkowski inequality":
\begin{equation} \label{eq:QBM}
\mm(Z_t(A,B))^{\frac{1}{N}} \geq \frac{1}{Q^{\frac{1}{N}}} \brac{(1-t) \mm(A)^{\frac{1}{N}} + t \mm(B)^{\frac{1}{N}} } ,
\end{equation}
for all Borel sets $A,B \subset X$ of finite positive measure. On the other hand, it was shown by Juillet in \cite{Juillet-SubRiemannianNotCD} that on any strictly sub-Riemannian manifold $M$ equipped with its sub-Riemannian Carnot--Carath\'eodory metric $\d$ and any positive smooth measure $\mm$, and for any $\eps > 0$, there exist Borel sets $A,B \subset M$ of finite positive measure and $t \in (0,1)$, so that:
\[
\frac{\mm(B)}{\mm(A)} \in [1-\eps,1+\eps]  ~,~ \mm(Z_t(A,B)) \leq \frac{1}{2^{\mathcal{N}-n}} (1+\eps) \mm(A) . 
\]
Juxtaposing this with (\ref{eq:QBM}), it follows that necessarily $Q \geq 2^{\mathcal{N}-n}$, and hence the value $Q = 2^{N-n}$ is optimal in both our $\QCD(Q,0,N)$ condition and in Juillet's construction whenever $\mathcal{N} = N$ (note that we always have $N \geq \mathcal{N}$ as a consequence of \cite[Theorem 5]{BarilariRizzi-Interpolation}). As a side note, we mention that Juillet's construction moreover guarantees that $\diam(A \cup B) < R$ for any given $R > 0$, which shows that $(M,\d,\mm)$ as above does not satisfy $\CD(K',N')$ for any $K',N' \in \R$. 

Note that (\ref{eq:QBM}) with $Q = 2^{N-n}$ in the ideal sub-Riemannian setting and with $Q=4$ ($N=n+2$) for corank $1$ Carnot groups, follows immediately by Jensen's inequality from the Brunn--Minkowski inequalities of Barilari--Rizzi \cite[Theorem 9]{BarilariRizzi-Interpolation} and Balogh--Krist\'aly--Sipos \cite[Theorem 4.2]{BKS-CorankOneCarnot}, respectively:
\begin{equation} \label{eq:BM}
\mm(Z_t(A,B))^{\frac{1}{n}} \geq (1-t)^{\frac{N}{n}} \mm(A)^{\frac{1}{n}} + t^{\frac{N}{n}} \mm(B)^{\frac{1}{n}} .
\end{equation}
Juillet's construction therefore demonstrates the optimality of (\ref{eq:BM}) not only when one of the sets degenerates to a point (as e.g. in \cite{BarilariRizzi-Interpolation}), but also for sets of equal measures.

\subsection{Equivalent characterization of the $\QCD$ condition}

Finally, we conclude this work by mentioning an essentially equivalent characterization of the $\QCD$ condition which highlights again the connection to the $\CD$ definition. The simplest case to examine is when $K=0$. 

\begin{definition}[$\QCD_{\mm}(Q,0,N)$]
A Monge space $(X,\d,\mm)$ is said to satisfy the $\QCD_{\mm}(Q,0,N)$ condition, $Q \geq 1$, $N \in (1,\infty)$,  if for all 
$\mu_0,\mu_1 \in P_c(X)$ with $\mu_0,\mu_1 \ll \mm$, there exist a family of Borel measures $(\mm_t)_{t \in [0,1]}$ with $\mm \leq \mm_t \leq Q \mm$ on $\supp \mu_t$, so that the $W_2$ geodesic $(\mu_t)$ satisfies the $\CD(0,N)$ interpolation inequality \emph{with respect to $(\mm_t)$} --  namely, denoting $\tilde \rho_t := \frac{d\mu_t}{d\mm_t}$, we have for all $t \in (0,1)$:
\begin{equation} \label{eq:QCDm}
\tilde \rho^{-\frac{1}{N}}_t(\gamma_t) \geq (1-t) \tilde \rho^{-\frac{1}{N}}_0(\gamma_0) + t \tilde \rho^{-\frac{1}{N}}_1(\gamma_1) \;\;
 \text{for $\nu$-a.e. $\gamma \in \geo(X,\d)$}. 
 \end{equation}
 \end{definition} 

We claim that this definition is equivalent to the original $\QCD(Q,0,N)$ definition on Monge spaces satisfying $\MCP(K',N')$ for some $K' \in \R$ and $N' \in (1,\infty)$. Indeed, since $\rho_t := \frac{d\mu_t}{d \mm}$ satisfies $\tilde \rho_t \leq \rho_t \leq Q \tilde \rho_t$, if $\QCD_\mm(Q,0,N)$ holds then clearly $\QCD(Q,0,N)$ holds as well by passing from (\ref{eq:QCDm}) to (\ref{eq:def-QCD}). In the other direction, the $\MCP(K',N')$ condition guarantees that given $(\mu_t)$ as above, we may choose versions of the densities $\rho_t := \frac{d\mu_t}{d \mm}$ so that $(0,1) \ni t \mapsto \rho_t(\gamma_t)$ is continuous and upper semi-continuous at the end-points for $\nu$-a.e. $\gamma \in \geo(X,\d)$ (see \cite[Corollary 9.5 and Remark 9.9]{CavallettiEMilman-LocalToGlobal}). 
If the space in addition satisfies $\QCD(Q,0,N)$, then by considering all rational $t \in (0,1)$ and employing the latter continuity, it follows that there is a subset $G$ of geodesics $\gamma$ having full $\nu$-measure, so that $1/\rho_t(\gamma_t)$ satisfies (\ref{eq:def-QCD}) for all $t \in (0,1)$ and is therefore almost a $\QCD(Q,0,N+1)$ density on $[0,1]$ (this is where the assumption $K=0$ comes in handy) -- it satisfies all requirements but is only lower semi-continuitous at the end points $t \in \{0,1\}$. Nevertheless, inspecting the proof of Proposition \ref{prop:QCD-CD}, it follows that there exists a continuous $\CD(0,N+1)$ density $f_\gamma : [0,1] \rightarrow \R_+$ so that $1/\rho_t(\gamma_t) \leq f_\gamma(t) \leq Q /\rho_t(\gamma_t)$ for all $t \in (0,1)$, and also $1/\rho_t(\gamma_t) \leq f_\gamma(t)$ for $t \in \{0,1\}$ by lower semi-continuity. Since the space is Monge, one knows that there is a subset $H$ of geodesics of full $\nu$-measure for which $H \ni \gamma \mapsto \gamma_t$ is injective for all $t \in [0,1]$ (see e.g. \cite[Corollary 6.15]{CavallettiEMilman-LocalToGlobal}). Consequently, denoting $\xi_0 = \xi_1 \equiv 1$ and $\xi_t(\gamma_t) := f_\gamma(t) \rho_t(\gamma_t) \in [1,Q]$ for $\gamma \in G \cap H$ and $\xi_t = 0$ elsewhere for $t \in (0,1)$, it follows that $\xi_t$ is well defined, and standard arguments imply that $\xi_t$ is measurable. We can now define $\mm_t = \xi_t \mm$, and it readily follows that $\mm \leq \mm_t \leq Q \mm$ on $\supp \mu_t$. Since $\tilde \rho_t = \rho_t / \xi_t$ so that $1/\tilde \rho_t(\gamma_t) = f_\gamma(t)$ for $\gamma \in G \cap H$ and $t \in (0,1)$ and $1/\tilde \rho_t(\gamma_t) = 1 / \rho_t(\gamma_t)$ for $t\in \{0,1\}$, it follows that for all $t \in (0,1)$, for $\nu$-a.e. $\gamma \in \geo(X,\d)$:
\[
\tilde \rho^{-\frac{1}{N}}_t(\gamma_t) = f_\gamma^{\frac{1}{N}}(t)\geq (1-t) f_\gamma^{\frac{1}{N}}(0) + t f_\gamma^{\frac{1}{N}}(1) \geq
(1-t) \tilde \rho^{-\frac{1}{N}}_0(\gamma_0) + t \tilde \rho^{-\frac{1}{N}}_1(\gamma_1)  ,
\]
and so we confirm that  (\ref{eq:QCDm}) is satisfied, i.e. that the space verifies $\QCD_{\mm}(Q,0,N)$.

\setlinespacing{1.0}
\setlength{\bibspacing}{2pt}
\bibliography{../../../ConvexBib}
\end{document}